\documentclass[12pt,reqno]{amsart}
\usepackage{amsfonts}
\usepackage{a4wide}

\numberwithin{equation}{section}

\newtheorem{theorem}{Theorem}[section]
\newtheorem{proposition}[theorem]{Proposition}
\newtheorem{corollary}[theorem]{Corollary}
\newtheorem{lemma}[theorem]{Lemma}

\theoremstyle{definition}

\theoremstyle{remark}
\newtheorem{remark}[theorem]{Remark}

\begin{document}

\title[ threshold solutions
 ]
{
  Threshold solutions for the focusing
    $L^{2}$ -supercritical
    NLS Equations
 }

 \author{Qing Guo
}

\address{Academy of Mathematics and Systems Science, Chinese Academy of Sciences, Beijing 100190, P.R. China}

\email{guoqing@amss.ac.cn}

%\address{School of  Mathematics and Statistics, Central China Normal University, Wuhan,P.R. China}

%\email{sjpeng@mail.ccnu.edu.cn}

%\address{ School of Mathematics,Statistics and ComputerScience, The  University of New England Armidale, NSW 2351, Australia}

%\email{syan@turing.une.edu.au}

%\thanks{This work is partially supported by ARC}
\begin{abstract}
We investigate the $L^2$-supercritical and $\dot{H}^1$-subcritical nonlinear Schr\"{o}dinger equation
in $H^1$. In \cite{G1} and \cite{yuan}, the mass-energy quantity
$M(Q)^{\frac{1-s_{c}}{s_{c}}}E(Q)$ has been shown to be a threshold
for the dynamical behavior of solutions of the equation. In the present paper,
we study the dynamics at the critical level $M(u)^{\frac{1-s_{c}}{s_{c}}}E(u)=M(Q)^{\frac{1-s_{c}}{s_{c}}}E(Q)$
and classify the corresponding solutions using modulation theory,
non-trivially generalize  the results obtained in \cite{holmer3}
for
the 3D cubic Schr\"{o}dinger equation.

\end{abstract}

\maketitle


\section{Introduction}

We consider the following Cauchy problem of a nonlinear Schr\"{o}dinger equation %in dimension

\begin{equation}\label{1.1}
\left\{ \begin{aligned}
         \ iu_{t}+\Delta u+|u|^{p-1}u&=0,\ \ \ (x,t)\in R^{N}\times R, \\
                  \ u(x,0)&=u_{0}(x)\in H^{1}(R^{N}).
                          \end{aligned}\right.
                          \end{equation}

It is well known from  \cite{JG1} and \cite{TC} that, equation \eqref{1.1} is locally well-posed in $ H^{1}.$
That is for  $u_{0}\in H^{1},$ there exist $0<T\leq\infty$ and a unique solution $u(t)\in C([0,T);H^{1})$ to   \eqref{1.1}.
When  $T=\infty,$   we say that the solution is positively  global; while on the other hand,  we  have
 $\lim_{t\uparrow T}\|\nabla u(t)\|_{2}\rightarrow\infty$   and call that this solution blows up in finite positive time.
Solutions of
 \eqref{1.1} admits the following conservation laws in energy space   $H^{1}:$
\begin{align*}
L^{2}-norm:\ \ \ \ M(u)(t)&\equiv \int|u(x,t)|^{2}dx=M(u_{0});\\
Energy:\ \ \ \ E(u)(t)&\equiv \frac{1}{2}\int|\nabla u(x,t)|^{2}dx-\frac{1}{p+1}\int|u(x,t)|^{p+1}dx=E(u_{0});\\
Momentum:\ \ \ \ P(u)(t)&\equiv Im\int\overline{u}(x,t)\nabla u(x,t)dx=P(u_{0}).
\end{align*}

Note that equation \eqref{1.1} is invariant under the scaling $u(x,t)\rightarrow\lambda^\frac{2}{p-1}u(\lambda x,\lambda^2t)$
which also leaves the homogeneous Sobolev norm
$\dot{H}^{s_c}$ invariant with  $s_c=\frac{N}{2}-\frac{2}{p-1}.$
Other scaling invariant quantities are  $\|\nabla u\|_{2}\|u\|_{2}^{\frac{1-s_{c}}{s_{c}}}$
and  $E(u)M(u)^{\frac{1-s_{c}}{s_{c}}}.$
It is classical from the conservation of the energy and the $L^{2}$ norm that for  $s_c<0$, the equation is subcritical
and all  $ H^{1}$ solutions are global and $ H^{1}$ bounded. The smallest power for which blow up may occur is $p=1+\frac{4}{N}$
which is referred to as the $L^{2}$ critical case  corresponding to  $s_c=0$ (see \cite{RTG} \cite{SK2}). The case  $0<s_c<1$
(equivalent to $1+\frac4N<p<1+\frac4{N-2}$) is called the
$L^{2}$ supercritical and  $H^{1}$ subcritical case. In this paper, we are concerning  with the case  $0<s_c<1.$

We say that $(q,r) $ is $\dot{H}^{s}(\mathbb{R}^{N})$-admissible ($0\leq s\leq1$) denoted by
 $(q,r)\in\Lambda_{s}$ if
$$\frac{2}{q}+\frac{N}{r}=\frac{N}{2}-s,\ \ \ \frac{2N}{N-2s}<r<\frac{2N}{N-2}.$$
This is associated to the well-known Strichartz's estimates:
for any $\varphi\in\dot{H}^{s}, f(x,t)\in L_t^{q}L_x^{r}$ and any admissible pair
$(q,r),(\gamma,\rho) \in\Lambda_{s}$, we have
\begin{align}\label{strichartz}
\|e^{it\Delta}\varphi\|_{L_t^{q}L_x^{r}}\leq C\|\varphi\|_{\dot{H}^{s}},\ \ \
\|Gf\|_{L_t^{\gamma'}L_x^{\rho'}}\leq C \|f\|_{L_t^{q}L_x^{r}},
\end{align}
where $\frac{1}{\rho'}+\frac1\rho=\frac{1}{\gamma'}+\frac1\gamma=1,$ and
$Gf(t,
x)\equiv\int\limits_{t_0}^{t}e^{i(t-s)\Delta}f(s)ds.$
We  define the following Srichartz norm
$$\|u\|_{S(\dot{H}^{s})}=\sup_{(q,r)\in\Lambda_{s}}\|u\|_{L_t^qL_x^r}
 $$
 and recall the following properties for the Cauchy problem \eqref{1.1},
 which can be found in \cite{yuan}:
\begin{proposition}\label{sd}
(Small initial data). Let $\|u_{0}\|_{\dot{H}^{s_{c}}}\leq A$, then
there exists  $\delta_{sd}=\delta_{sd}(A)>0$ such that if
 $\|e^{it\Delta}u_{0}\|_{S(\dot{H}^{s_{c}})}\leq \delta_{sd}, $ then $ u $ solving \eqref{1.1} is global and
\begin{eqnarray}
&&\|u\|_{S(\dot{H}^{s_{c}})}\leq
2\|e^{it\Delta}u_{0}\|_{S(\dot{H}^{s_{c}})},\\
&&\|D^{s_{c}}u\|_{S(L^{2})}\leq
2c\|u_{0}\|_{\dot{H}^{s_{c}}}.
\end{eqnarray}

\end{proposition}

\begin{remark}\label{ss}
Note that by
 Strichartz's estimates, the hypotheses are satisfied if
 $\|u_{0}\|_{\dot{H}^{s_{c}}}\leq C\delta_{sd}. $
 Furthermore, by the result obtained by \cite{yuan}, the uniform bound
 of $\dot{H}^{s_{c}}$-norm of the solution $u$ to \eqref{1.1} implies
 $u(t)$ scatters as $t\rightarrow\pm\infty$.
\end{remark}

\begin{proposition}\label{wave}
(Existence of wave operators). Suppose that $\psi^{+}\in H^1$ and
\begin{equation}\label{3.6}
\frac{1}{2^{}}||\nabla\psi^{+}||_{2}^{2}M(\psi^{+})^{\frac{1-s_{c}}{s_{c}}}<
E(Q)^{}M(Q)^{\frac{1-s_{c}}{s_{c}}}.
\end{equation}
Then there exists $v_{0}\in H^1$ such that v solves \eqref{1.1} with
initial data $v_0$ globally in $H^{1}$ with
$$\|\nabla
v(t)\|_{2}^{}\|v_{0}\|_{2}^{\frac{1-s_{c}}{s_{c}}}<\|\nabla
Q\|_{2}^{}\|Q\|_{2}^{\frac{1-s_{c}}{s_{c}}},M(v)=\|\psi^{+}\|_{2}^{2},E[v]=\frac{1}{2}\|\nabla\psi^{+}\|_{2}^{2},$$
and
$$\lim_{t\rightarrow+\infty}\|v(t)-e^{it\Delta}\psi^{+}\|_{H^1}=0.$$
Moreover, if
$\|e^{it\Delta}\psi^{+}\|_{S(\dot{H}^{s_{c}})}\leq \delta_{sd}$, then
$$\|v_{0}\|_{\dot{H}^{s_{c}}}\leq2\|\psi^{+}\|_{\dot{H}^{s_{c}}}\
\mathrm{and} \
 \|v\|_{S(\dot{H}^{s_{c}})}\leq2\|e^{it\Delta}\psi^{+}\|_{S(\dot{H}^{s_{c}})}.$$
$$\|D^{s}v\|_{S(L^{2})}\leq c\|\psi^{+}\|_{\dot{H}^s}, 0\leq s\leq1.$$
\end{proposition}

\begin{proposition}\label{perturb}
(long time perturbation theory). $\forall$ $A\geq 1$, there exists
$\epsilon_{0}=\epsilon_{0}(A)$, $c=c(A)\gg1$ such that if
$u=u(x,t)\in H^{1}$ satisfy
$$iu_{t}+\Delta
u+|u|^{p-1}u=0.$$
 $\tilde{u}=\tilde{u}(x,t)\in H^{1}$ ,define
$$e =i\tilde{u}_{t}+\Delta
\tilde{u}+|\tilde{u}|^{p-1}\tilde{u}$$
with $\|\tilde{u}\|_{S(\dot{H}^{s_{c}})}\leq
A$. If
\begin{equation*}
\|e\|_{S(\dot{H}^{s_{c}})}\leq \epsilon_{0},
\end{equation*}
$$\|e^{i(t-t_{0})\Delta}(u(t_{0})-\tilde{u}(t_{0}))\|_{S(\dot{H}^{s_{c}})}\leq
\epsilon_{0},$$ then
$$\|u\|_{S(\dot{H}^{s_{c}})}\leq c=c(A)<\infty. $$
\end{proposition}

For the 3D cubic nonlinear Schr\"{o}dinger equation
with  $s_c=\frac{1}{2}$ and $p=3,$ there have been several results on  either scattering  or blow-up solutions.
In \cite{radial}, \cite{nonradial} and \cite{holmer10}, Roudenko and Holmer have shown that
$M(Q)E(Q)$ plays an important role in the dynamical behavior of solutions of equation \eqref{1.1} with $p=3$ and $N=3$.
The authors in \cite{yuan} and \cite{G1} extended their results to the general
$L^2$-supercritical and $\dot{H}^1$-subcritical case and showed that
$M(Q)^{\frac{1-s_{c}}{s_{c}}}E(Q)$ is an threshold for the dynamics in the following sense:
Let $u$ be a solution of \eqref{1.1} %with the initial data $u_0$
satisfying
$M(u)^{\frac{1-s_{c}}{s_{c}}}E(u)<M(Q)^{\frac{1-s_{c}}{s_{c}}}E(Q)$.
Then if  $\|\nabla u_{0}\|_{2}\|u_{0}\|_{2}^{\frac{1-s_{c}}{s_{c}}}<\|\nabla Q\|_{2}\|Q\|_{2}^{\frac{1-s_{c}}{s_{c}}}$,
we have $T_+=T_-=\infty$ and $\|u\|_{S(\dot{H}^{s_{c}})}<\infty.$
On the other hand, if
$\|\nabla u_{0}\|_{2}\|u_{0}\|_{2}^{\frac{1-s_{c}}{s_{c}}}>\|\nabla Q\|_{2}\|Q\|_{2}^{\frac{1-s_{c}}{s_{c}}}$,
then either  $u(t)$ blows up in finite forward time,  or $u(t)$ is forward global and there exists a
time sequence $t_{n}\rightarrow\infty$ such that  $\|\nabla u(t_{n})\|_{2}\rightarrow\infty.$
A similar statement holds for negative time.
Our goal in this paper is to give a classification of solutions of the solution of \eqref{1.1}
with the critical level:
\begin{align}\label{1.3}
M(u)^{\frac{1-s_{c}}{s_{c}}}E(u)=M(Q)^{\frac{1-s_{c}}{s_{c}}}E(Q)
\end{align}
 extending the very recent results obtained in \cite{holmer3} for the particular case
with $p=3$ and $N=3$.
 The idea in this paper follows from Kenig-Merle \cite{M2} for the energy-critical NLS.

In this paper we obtain the following results:
\begin{theorem}\label{th2}
There exist two radial solutions $Q^+$ and $Q^-$ of \eqref{1.1}
with initial data $Q^\pm_0\in\cap_{s\in\mathbb R}H^s(\mathbb R^N)$ and satisfy\\
(a) $M(Q^+)=M(Q^-)=M(Q)$, $E(Q^+)=E(Q^-)=E(Q)$, $[0,+\infty)$ is in the domain of
the definition of $Q^\pm$ and there exists $e_0>0$ such that
$$\|Q^\pm(t)-e^{i(1-s_c)t}Q\|_{H^1}\leq Ce^{-e_0t},\ \ \ \forall\ \ t\geq0;$$
(b) $\|\nabla Q^-_0\|_2<\|\nabla Q\|_2$, $Q^-$ is globally defined and scatters for negative time;\\
(c) $\|\nabla Q^+_0\|_2>\|\nabla Q\|_2$, and the negative time of existence of $Q^+$ is finite.
\end{theorem}

\begin{theorem}\label{th3}
Let $u$ be a solution of \eqref{1.1} satisfying \eqref{1.3}.\\
(a) If $\|\nabla u_0\|_2\|u_0\|^{\frac{1-s_{c}}{s_{c}}}_2<\|\nabla Q\|_2\|Q\|^{\frac{1-s_{c}}{s_{c}}}_2$,
then either $u$ scatters or $u=Q^-$ up to the symmetries;\\
(b) If $\|\nabla u_0\|_2\|u_0\|^{\frac{1-s_{c}}{s_{c}}}_2=\|\nabla Q\|_2\|Q\|^{\frac{1-s_{c}}{s_{c}}}_2$,
then $u=e^{i(1-s_c)t}Q$ up to the symmetries;\\
(c) If $\|\nabla u_0\|_2\|u_0\|^{\frac{1-s_{c}}{s_{c}}}_2>\|\nabla Q\|_2\|Q\|^{\frac{1-s_{c}}{s_{c}}}_2$,
and $u_0$ is radial or of finite variance, then either the interval of existence of $u$ is of finite or $u=Q^+$
up to the symmetries.
\end{theorem}

\begin{remark}\label{p}
Equation \eqref{1.1} admits the Galilean invariance:
If $u(x,t)$ is a solution of \eqref{1.1}, then for any $\xi_0\in\mathbb R^N$,
$w(x,t)\equiv u(x-\xi_0t,t)e^{i\frac{\xi_0}2\cdot(x-\frac{\xi_0}2t)}$ also satisfies the equation \eqref{1.1}.
Recall from the Appendix of \cite{G1},
taking the Galilean transform with $\xi_0=-P(u)/M(u)$  into account, we get a solution with zero momentum which is
the minimal energy solution $v$ among all Galilean transformations of the solution $u$ of \eqref{1.1}.
Precisely, $M(v)=M(u), E(v)=E(u)-\frac12\frac{P(u)^2}{M(u)}$ and $\|v_0\|^2_2=\|u_0\|^2_2-\frac12\frac{P(u_0)^2}{M(u_0)}$.
Applying  Theorem \ref{th3} and the results obtained in \cite{yuan} and \cite{G1} to $v$, we indeed obtain that
\begin{theorem}\label{th4}
Let $u$ be a solution of \eqref{1.1} satisfying
\begin{align*}
M(u)^{\frac{1-s_{c}}{s_{c}}}E(u)-\frac12P(u)^2\leq M(Q)^{\frac{1-s_{c}}{s_{c}}}E(Q).
\end{align*}
Then,\\
(a) If $\|\nabla u_0\|_2\|u_0\|^{\frac{1-s_{c}}{s_{c}}}_2-P(u)^2<\|\nabla Q\|_2\|Q\|^{\frac{1-s_{c}}{s_{c}}}_2$,
then either $u$ scatters or $u=Q^-$ up to the symmetries;\\
(b) If $\|\nabla u_0\|_2\|u_0\|^{\frac{1-s_{c}}{s_{c}}}_2-P(u)^2=\|\nabla Q\|_2\|Q\|^{\frac{1-s_{c}}{s_{c}}}_2$,
then $u=e^{i(1-s_c)t}Q$ up to the symmetries;\\
(c) If $\|\nabla u_0\|_2\|u_0\|^{\frac{1-s_{c}}{s_{c}}}_2-P(u)^2>\|\nabla Q\|_2\|Q\|^{\frac{1-s_{c}}{s_{c}}}_2$,
and $u_0$ is radial or of finite variance, then either the interval of existence of $u$ is of finite or $u=Q^+$
up to the symmetries.
\end{theorem}
\end{remark}

%%%%%%%%%%%%%%%%%%%%%%%%%%%%%%%%%%%%%%%%%%%%%%%%%%%%%%%%%%%%%%%%%%%%%%%%%%%%%%%%%%%%%%%%%%%%%%%%%%%%%%%%%%%%%%%%%%%%%%%%%%%%%%%%%%%%%%%%%%%
%%%%%%%%%%%%%%%%%%%%%%%%%%%%%%%%%%%%%%%%%%%%%%%%%%%%%%%%%%%%%%%%%%%%%%%%%%%%%%%%%%%%%%%%%%%%%%%%%%%%%%%%%%%%%%%%%%%%%%%%%%%%%%%%%%%%
%%%%%%%%%%%%%%%%%%%%%%%%%%%%%%%%%%%%%%%%%%%%%%%%%%%%%%%%%%%%%%%%%%%%%%%%%%%%%%%%%%%%%%%%%%%%%%%%%%%%%%%%%%%%%%%%%%%%%%%%%%%%%%%%%重新再写文章创新点及难点
The outline of this paper is as follows.
In section 2, we recall some properties of the ground state $Q$ %, small data theory for the Cauchy problem \eqref{1.1}
and analyze the linearized equation associated to \eqref{1.1} near $e^{i(1-s_c)t}Q$.
 %the spectral properties of the linearized Schr\"odinger operator.
In section 3, we construct a family of approximate solutions using the descrete spectrum of the linearized operator
and produce candidates for the special solutions $Q^+$ and $Q^-$. Then in section 4, %using arguments in \cite{W2},
we discuss the modulational stability near $Q$, which is important for our study of solutions with initial data from
part (a) and (c) in Theorem \ref{th3}. This is done
in sections 5 and 6 respectively. In section 7, we establish the uniqueness of special solutions by analyzing the linearized equation
and finally finish the proof of the classification of solution in the critical
level. 

This paper is a non-trivial generalization of \cite{holmer3}, which deals with the 3D cubic Schr\"odinger equations.
First of all,  quite different from the case $p=3, N=3$ considered in \cite{holmer3},  our 
 $p$ is not  an integer when $N\geq 4$, since $1+\frac4N<p<1+\frac4{N-2}$. This mainly
 brings two  difficulties for our study as follows.
  On the one hand, it is not enough to  consider the problem just in the space $C_b(I;H^1)$ as the authors did in \cite{holmer3},
   where $I\subset\mathbb R$ is
  a time interval.
Instead, we should also
work on the Strichartz space $L^{\frac{4(p+1)}{N(p-1)}}(I;L^{p+1}(\mathbb R^N))$ and use the corresponding Strichartz's estimates associated to the
Schr\"odinger operator $e^{it(\Delta-(1-s_c))}$, which is just like the classical Strichartz's estimates.
On the other hand, the general case require more sophisticated analyzing on the spectral properties of the linearized Schr\"odinger operators.
Moreover,
because of the technical difficulties,  we cannot directly use the linearized equation near
 $e^{it}\tilde Q$ with $\tilde Q$ solving the elliptic equation
$-\Delta Q+Q-Q^p=0$ as the authors did in \cite{holmer3}; while instead, we linearize the equation near $e^{i(1-s_c)t}Q$,
where $Q$ solves $-\Delta Q+(1-s_c)Q-Q^p=0$. %For example, we will obtain \eqref{phiq} from this property in Section 2.

In this paper,
we denote the
Sobolev spaces $ H^{1}(\mathbb{R}^{N})$ and $ W^{m,p}(\mathbb{R}^{N})$ as $H^{1}$ and
$ W^{m,p}$  for short, and the  $L^{p}$ norm as $\|\cdot\|_{p}.$
$C$ is denoted variant absolute constants only depending on $N$ and $p$.
%Also for convenience, we will use the notation $C,$ except for some specifications,  standing for the variant absolute constants.

\section{Preliminaries }

\subsection{Properties of the ground state}

Weinstein in \cite{W1} proved that the sharp constant  $C_{GN}$  of Gagliardo-Nirenberg inequality for $0<s_{c}<1$

\begin{equation}\label{2.1}
\|u\|^{p+1}_{L^{p+1}(R^{N})}\leq C_{GN}\|\nabla
u\|_{L^{2}(R^{N})}^{\frac{N(p-1)}{2}}\|u\|_{L^{2}(R^{N})}^{2-\frac{(N-2)(p-1)}{2}}
\end{equation}
is achieved by the unique minimizer $u=Q,$ where $Q$ is the ground state of
\begin{align}\label{Q}
-(1-s_{c})Q+ \Delta Q+|Q|^{p-1}Q=0,
\end{align}
which is radial, smooth, positive, exponentially decaying at infinity.
In other words, if
\begin{align}\label{chara1}
\|u\|^{p+1}_{L^{p+1}(R^{N})}= C_{GN}\|\nabla
u\|_{L^{2}(R^{N})}^{\frac{N(p-1)}{2}}\|u\|_{L^{2}(R^{N})}^{2-\frac{(N-2)(p-1)}{2}},
\end{align}
then, there exists $\lambda_0\in\mathbb C$ and $x_0\in\mathbb R^N$ such that
$u(x)=\lambda_0Q(x+x_0)$.

Applying the concentration-compactness principle, the characterization of $Q$ yields the following
proposition:
\begin{proposition}\label{charaQ}
(\cite{lions}) There exists a function $\epsilon(\rho),$ defined for small $\rho>0$
 such that\\ $\lim_{\rho\rightarrow 0}\epsilon(\rho)=0,$ such that for all  $u\in H^{1}$ with
\begin{equation*}\label{}
\left|\|u\|_{p+1}-\|Q\|_{p+1}\right|+\left|\|u\|_{2}-\|Q\|_{2}\right|+\left|\|\nabla u\|_{2}-\|\nabla Q\|_{2}\right|
\leq \rho,
\end{equation*}
there exist $\theta_{0}\in\mathbb{R}$ and $x_{0}\in\mathbb{R}^{N}$ such that
\begin{equation*}\label{}
\left\|u-e^{i\theta_{0}}Q(\cdot-x_{0})\right\|_{H^{1}}
\leq \epsilon(\rho).
\end{equation*}
\end{proposition}

Using Pohozhaev identities we can get the following identities without difficulty:
\begin{equation}\label{Q=}
\|Q\|_{2}^{2}=\frac{2}{N}\|\nabla Q\|_{2}^{2},\ \ \
\|Q\|^{p+1}_{p+1}=\frac{2(p+1)}{N(p-1)}\|\nabla Q\|_{2}^{2}=\frac{(p+1)}{(p-1)}\|Q\|_{2}^{2},
\end{equation}
\begin{equation}\label{2.2}
E(Q)=\frac{N(p-1)-4}{2N(p-1)}\|\nabla Q\|_{2}^{2}=\frac{N(p-1)-4}{4(p-1)}\|Q\|_{2}^{2}=\frac{N(p-1)-4}{4(p+1)}\|Q\|^{p+1}_{p+1},
\end{equation}
and $C_{GN}$  can be  expressed by
\begin{equation}\label{2.3}
 C_{GN}=\frac{\|Q\|^{p+1}_{p+1}}{\|\nabla Q\|_{2}^{\frac{N(p-1)}{2}}\|Q\|_{2}^{2-\frac{(N-2)(p-1)}{2}}}.
 \end{equation}

By the  $H^{1}$ local theory \cite{TC} , there exist $-\infty\leq T_{-}<0<T_{+}\leq+\infty$ such that $(T_{-},T_{+})$ is the maximal
time interval of existence for  $u(t)$ solving \eqref{1.1} , and if $T_{+}<+\infty$ then
$$\|\nabla u(t)\|_{2}\geq\frac{C}{(T_{+}-t)^{\frac{1}{p-1}-\frac{N-2}{4}}}\ \ \ as t\uparrow T_{+},$$
and a similar argument holds if  $-\infty< T_{-}.$  Moreover,
as a consequence of the continuity of the flow $u(t),$ we have the following dichotomy proposition :
\begin{proposition}\label{p21'}
 Let $u_{0}\in H^{1}(R^{N})$,  and let $I=(T_{-},T_{+})$ be the
maximal time interval of existence of $u(t)$ solving \eqref{1.1}
and suppose \eqref{1.3} holds.\\
(a) If
$
\|u_{0}\|_{2}^{\frac{1-s_{c}}{s_{c}}}\|\nabla u_{0}\|_{2}<\|Q\|_{2}^{\frac{1-s_{c}}{s_{c}}}\|\nabla Q\|_{2},
$
then $I=(-\infty,+\infty)$, i.e., the solution exists globally in
time, and for all time $t\in \mathbb{R},$
$
\|u(t)\|_{2}^{\frac{1-s_{c}}{s_{c}}}\|\nabla u(t)\|_{2}<\|Q\|_{2}^{\frac{1-s_{c}}{s_{c}}}\|\nabla Q\|_{2}.
$\\
(b) If $
\|u_{0}\|_{2}^{\frac{1-s_{c}}{s_{c}}}\|\nabla u_{0}\|_{2}=\|Q\|_{2}^{\frac{1-s_{c}}{s_{c}}}\|\nabla Q\|_{2},
$
then $u=e^{i(1-s_c)t}Q$ up to the symmetries.\\
(c) If
$
\|u_{0}\|_{2}^{\frac{1-s_{c}}{s_{c}}}\|\nabla u_{0}\|_{2}>\|Q\|_{2}^{\frac{1-s_{c}}{s_{c}}}\|\nabla Q\|_{2},
$
then for all $t\in I,$
$
\|u(t)\|_{2}^{\frac{1-s_{c}}{s_{c}}}\|\nabla u(t)\|_{2}>\|Q\|_{2}^{\frac{1-s_{c}}{s_{c}}}\|\nabla Q\|_{2}.
$

\end{proposition}
\begin{proof}
By rescaling, we can assume $M(u)=M(Q)$ and $E(u)=E(Q)$.
In fact, if $M(u)=\alpha M(Q)$, then we set $\lambda^{-2s_c}=\alpha$ and
 $\tilde u(x,t)=\lambda^{2/(p-1)}u(\lambda x,\lambda^2t)$.
 Thus, the assumption \eqref{1.3} implies that $M(\tilde u)=M(Q)$
 and $E(\tilde u)=E(Q)$.

 Case (b) is given by the variational characterization \eqref{chara1} and the uniqueness
 of solutions of \eqref{1.1}.
If Case (a) is false and suppose, by continuity,  there exists $t_1$ such that
$\|u(t_1)\|_2=\|Q\|_2$, then by Case (b) with the initial condition at $t=t_1$,
the equality holds for all times, which contradicts the condition at $t=0$.
Then Case (a) is true. We can prove Case (c) by similar arguments.

\end{proof}

\subsection{Properties of the linearized operator}

We consider a solution $u$ of \eqref{1.1} close to $e^{i(1-s_c)t}Q$ and write
$$u(x,t)=e^{i(1-s_c)t}(Q(x)+h(x,t)).$$
Explicitly, $h$ satisfies that
\begin{align}\label{linea2}
i\partial_th+\Delta h-(1-s_c)h=-S(h),
\end{align}
where
\begin{align}\label{S}
S(h)\equiv|Q+h|^{p-1}(Q+h)-Q^p\equiv Vh-R(h)
\end{align}
with the linear part $Vh$ of $h$ defined by
\begin{align}\label{V}
Vh\equiv pQ^{p-1}h_1+iQ^{p-1}h_2
\end{align}
and $R(h)=O(Q^{p-2}|h|^2+|h|^{p-1}h)$ with its expression:
\begin{align}\label{2.9}
R(h)\equiv Q^p+pQ^{p-1}h_1+iQ^{p-1}h_2-|Q+h|^{p-1}(Q+h).
\end{align}
Similar to the   Strichartz's estimates associated to the classical Schr\"odinger operator
$e^{it\Delta}$, we also have the same Strichartz inequalities as \eqref{strichartz} associated to the
 little modified Schr\"odinger operator $e^{it(\Delta-(1-s_c))}$.
 In fact, $e^{it(\Delta-(1-s_c))}$ is no other than  $e^{it(1-s_c)}e^{it\Delta}$ and should keep the estimates \eqref{strichartz}.
 Also,  one can refer to  \cite{Kenji Yajima} for this result.
Furthermore, by the expression of  $Vh$ and $R(h)$,  we  have the following elementary estimates:
For any time interval $I$ with $|I|<\infty$, if we set $\tilde r=p+1$ and $\frac2{\tilde q}=N(\frac12-\frac1{\tilde r})$, then
from H\"older inequality and in view of the exponentially
 decay of $Q$ at infinity, we have
 \begin{align}\label{eV}
 \|Vh\|_{L^{\tilde q'}(I;W^{1,\tilde r'})}\leq C|I|^{\frac1{\tilde q'}-\frac1{\tilde q}}\|h\|_{L^{\tilde q}(I;W^{1,\tilde r})},
 \end{align}
  \begin{align}\label{eS}
 \|S(h)\|_{L^{\tilde q'}(I;W^{1,\tilde r'})}\leq C|I|^{\frac1{\tilde q'}-\frac1{\tilde q}}\|h\|_{L^{\tilde q}(I;W^{1,\tilde r})}
 (1+\|h\|^{p-1}_{L^{\infty}(I;H^{1})}),
 \end{align}
  \begin{align}\label{eR}
& \|R(h)-R(g)\|_{L^{\tilde q'}(I;L^{\tilde r'})}\\ \nonumber
 &\leq C|I|^{\frac1{\tilde q'}-\frac1{\tilde q}}
 \|h-g\|_{L^{\tilde q}(I;L^{\tilde r})} \Big(\|h\|_{L^{\tilde q}(I;L^{\tilde r})}+\|g\|_{L^{\tilde q}(I;L^{\tilde r})}
 +\|h\|^{p-1}_{L^{\infty}(I;H^{1})}+\|g\|^{p-1}_{L^{\infty}(I;H^{1})}\Big)
 \end{align}
 and
   \begin{align}\label{eR2}
 &\|\nabla R(h)-\nabla R(g)\|_{L^{\tilde q'}(I;W^{1,\tilde r'})}\\ \nonumber
 &\leq C|I|^{\frac1{\tilde q'}-\frac1{\tilde q}}
 \|h-g\|_{L^{\tilde q}(I;W^{1,\tilde r})} \Big(\|h\|_{L^{\tilde q}(I;W^{1,\tilde r})}+\|g\|_{L^{\tilde q}(I;W^{1,\tilde r})}
 +\|h\|^{p-1}_{L^{\infty}(I;H^{1})}+\|g\|^{p-1}_{L^{\infty}(I;H^{1})}\Big).
 \end{align}

Now, let $h_1=Re ~h,\ \  h_2=Im ~h$. If we identify $h=h_1+ih_2\in \mathbb C$ as an element
$(h_1,h_2)^T$ of $\mathbb R^2$, then $h$ is a solution of the equation
\begin{align}\label{linea}
\partial_th+\mathcal Lh=R(h),\ \ \
\mathcal L\equiv\left(
\begin{matrix}
0 & -L_-\\
L_+ & 0
\end{matrix}\right),
\end{align}
where the self-adjoint operators $L_+$and $L_-$ are defined by
\begin{align}\label{2.8}
L_+h_1\equiv-\Delta h_1+(1-s_c)h_1-pQ^{p-1}h_1,\ \ \
L_-h_2\equiv-\Delta h_2+(1-s_c)h_2-Q^{p-1}h_2.
\end{align}
By Weinstein \cite{W2}, we have the following spectral properties of the operator $\mathcal L$:
%which is defined on $L^2(\mathbb R^N)\times L^2(\mathbb R^N)$ and
\begin{proposition}\label{spectral}
Let $\sigma(\mathcal L)$ be the spectrum of the operator $\mathcal L$ defined on $L^2(\mathbb R^N)\times L^2(\mathbb R^N)$,
and let $\sigma_{ess}(\mathcal L)$ be its essential spectrum.
Then
$$\sigma_{ess}(\mathcal L)=\{i\xi: \xi\in\mathbb R, |\xi|\geq1\},\ \ \
\sigma(\mathcal L)\cap\mathbb R=\{-e_0, 0, e_0\}$$
with $e_0>0$. Furthermore, $e_0$ and $-e_0$ are simple eigenvalues of $\mathcal L$
with eigenfunctions $\mathcal Y_+, \mathcal Y_-=\overline{\mathcal Y}_+\in \mathcal S,$
and the null-space of $\mathcal L$ is spanned by the $N+1$ vectors $\partial_{x_j}Q$,
$j=1,\cdots,N$ and $iQ$.
\end{proposition}

By this proposition, if we let $\mathcal Y_1=Re\mathcal Y_+=Re\mathcal Y_-$
and $\mathcal Y_2=Im\mathcal Y_+=-Im\mathcal Y_-$,
then
\begin{align}\label{e_0}
L_+\mathcal Y_1=e_0\mathcal Y_2,\ \ L_-\mathcal Y_2=-e_0\mathcal Y_1,
\end{align}
and the null-space of $L_+$ is spanned by
the $N$ vectors $\partial_{x_j}Q$,
$j=1,\cdots,N$, while the null-space of $L_-$ is spanned by $Q$.
Moreover, also by \cite{W2}, we know that the operator $L_-$ is non-negative defined.

Define the linearized energy
\begin{align}\label{Phi}
\Phi(h)\equiv \frac {1-s_c}2\int|h|^2+\frac 12\int|\nabla h|^2-\frac 12\int Q^{p-1}(ph_1^2+h_2^2)
=\frac 12\int (L_+h_1)h_1+(L_-h_2)h_2.
\end{align}
Then $\Phi$ is conserved for solutions of the linearized equation
$\partial_th+\mathcal L h=0$.
By explicit calculation we have%(or refer to Appendix A of  \cite{holmer3} for a similar argument ),
\begin{align}\label{Phi2}
E(Q+h)=E(Q),\ \  M(Q+h)=M(Q)\ \ \ \ \Rightarrow\ \ \ \
|\Phi(h)|\leq c\|h\|_{p+1}^3.
\end{align}
In fact, $M(Q+h)=M(Q)$ yields that
\begin{align}\label{m=}
\int|h|^2=-2\int Qh_1.
\end{align}
On the other hand, from
 $E(Q+h)=E(Q)$, i.e.,
\begin{align*}
\frac12\int|\nabla Q+\nabla h|^2-\frac1{p+1}\int|Q+h|^{p+1}-\frac12\int|\nabla Q|^2+\frac1{p+1}\int|Q|^{p+1}=0,
\end{align*}
we obtain that
\begin{align*}
0=-\int\Delta Qh_1+\frac12\int|\nabla h|^2
-\int Q^{p}h_1-\frac12\int Q^{p-1}(ph_1^2+h_2^2)+O\left(\int Q^{p-2}|h|^3\right),
\end{align*}
which, combined with \eqref{m=} and \eqref{Phi}, gives \eqref{Phi2} by H\"older inequalities.

We now denote by $B(g,h)$ the bilinear symmetric form associated to $\Phi$ as
\begin{align}\label{B}
B(g,h)
=\frac 12\int (L_+g_1)h_1+(L_-g_2)h_2,
\end{align}
for all $g,h\in H^1$.
By Proposition \ref{spectral},
for any $h\in H^1$, we have
\begin{align}\label{2.14}
B(\partial_{x_j}Q,h)
=B(iQ,h)=0.
\end{align}
Furthermore, by \eqref{Q=}, we have
\begin{align}\label{phiq}
\Phi(Q)=\Big(\frac {1-s_c}2+\frac N4+\frac{p(p+1)}{2(p-1)}\Big)\|Q\|^2_2=-\frac{p^2-1}{4(p-1)}\|Q\|^2_2<0.
\end{align}
Thus, \eqref{phiq} and \eqref{2.14} imply immediately that
$\Phi(h)\leq0,$ for any $h\in span\{\partial_{x_j}Q,iQ,Q\}$, $j=1,\cdots ,N$.

Next, we are going  to find two subspaces of $H^1$ on which $\Phi$ is positive defined. In order to do this we consider
the following orthogonality relations:
\begin{align}\label{orth}
\int(\partial_{x_j}Q)h_1=\int Qh_2=0,
\end{align}
\begin{align}\label{orth2}
\int\Delta Qh_1=0,
\end{align}
\begin{align}\label{orth3}
\int\mathcal Y_1h_2=\int\mathcal Y_2h_1=0.
\end{align}
Let $G_\perp$ be the set of $h\in H^1$ satisfying \eqref{orth} and \eqref{orth2}
and $G_\perp^\prime$ be the set of $h\in H^1$ satisfying \eqref{orth} and \eqref{orth3}.
We then have the following:
\begin{proposition}\label{coercivity}
There exists a constant $c>0$ such that
\begin{align}\label{coer}
\Phi(h)\geq c\|h\|^2_{H^1},\ \ \ \forall h\in G_\perp\cap G_\perp^\prime.
\end{align}
\end{proposition}
The idea of the  proof of Proposition \ref{coercivity} follows  from \cite{W2} and \cite{holmer3}.
\begin{proof}
Firstly, when $h\in G_\perp$, we show the coercivity by two steps.

Step 1. We show $\Phi(h)\geq0$ for $h\in H^1$ satisfying \eqref{orth2}.
In fact, for $u\in H^1$, let
\begin{align}\label{I}
I(u)=\frac{\|\nabla u\|_2^{N(p-1)/2}\|u\|_2^{2-(N-2)(p-1)/2}
}{\|\nabla Q\|_2^{N(p-1)/2}\|Q\|_2^{2-(N-2)(p-1)/2}}-\frac{\|u\|^{p+1}_{p+1}
}{\|Q\|^{p+1}_{p+1}},
\end{align}
which can be shown  non-negative by \eqref{2.1} and \eqref{Q}.
By expansion of $I(Q+\alpha h)$ and in view of \eqref{orth2},  we
finally obtain that for $h\in H^1$ and $\alpha\in\mathbb R$,
\begin{align*}
&I(Q+\alpha h)=\left(1+\frac{N(p-1)}4\frac{\int|\nabla h_2|^2}{\int|\nabla Q|^2}\alpha^2
\right)
\Big(1+\frac{4-(N-2)(p-1)}2\frac{\int Qh_1}{\int Q^2}\alpha\\
&-
\frac{4(N-2)(p-1)-(N-2)^2(p-1)^2}{16}\Big(\frac{\int Qh_1}{\int Q^2}\Big)^2\alpha^2+
\frac{4-(N-2)(p-1)}4\frac{\int |h|^2}{\int Q^2}\alpha^2
\Big)\\
&-\left(1+(p+1)\frac{\int Q^ph_1}{\int Q^{p+1}}\alpha
+\frac{p+1}2\frac{\int Q^{p-1}(ph_1^2+h_2^2)}{\int Q^{p+1}}\alpha^2
\right)+O(\alpha^3).
\end{align*}
Since $I(Q)=0$ and $I(Q+\alpha h)\geq0$ for all real $\alpha$, the linear term in $\alpha$ should be
zero, and the quadratic term be nonnegative. Applying \eqref{Q=}, we obtain finally that
$$\frac{p-1}{\|Q\|_2^2}\Phi(h)\geq\frac{4(N-2)(p-1)-(N-2)^2(p-1)^2}{16}\Big(\frac{\int Qh_1}{\int Q^2}\Big)^2\geq0.$$

Step 2.  We show in this step that for $h$ fulfils \eqref{orth} and \eqref{orth2}
there exists some $c_*>0$ such that $\Phi(h)\geq c_*\|h\|_{H^1}^2$.
We denote $\Phi(h)=\Phi_1(h_1)+\Phi_2(h_2)$ with
$\Phi_1(h_1)\equiv\frac12\int(L_+h_1)h_1$, $\Phi_2(h_2)\equiv\frac12\int(L_-h_2)h_2$.
By step 1 and Proposition \ref{spectral}, $L_+$ is nonnegative on $\{\Delta Q\}^\perp$
and $L_-$ is nonnegative.
Following the arguments in \cite{W2} and \cite{holmer3}, we first show that
under the assumptions \eqref{orth} and \eqref{orth2}, there exists $c_1>0$
such that $\Phi_1(h_1)\geq c\|h_1\|_2^2$.
In fact, if not,  there exists a sequence $\{f_n\}$ of $H^1$ such that
\begin{align}\label{a6}
\lim_{n\rightarrow+\infty}\Phi_1(f_n)=0,\ \ \|f_n\|_2=1
\end{align}
and
$\int\Delta Qf_n=\int\partial_{x_j}Qf_n=0$ for
$j=1,\cdots,N$.
Thus we obtain that
\begin{align}\label{a8}
\frac12\int|\nabla f_n|^2=-\frac12+\frac p2\int Q^{p-1}f_n^2+o(1),
\end{align}
which implies that $\{f_n\}$ is bounded in $H^1$. Hence, up to a subsequence, we get that
there exists some $f_*\in H^1$ such that
$f_n\rightharpoonup f_*$ weakly in $H^1$
and $\frac p2\int Q^{p-1}f_n^2\rightarrow\frac p2\int Q^{p-1}f_*^2$.
Then by  \eqref{a8}, it follows that $\int Q^{p-1}f_*^2\geq \frac1p$, and so $f_*\neq0$.
From \eqref{a6} and the weak convergence of $\{f_n\}$, we get also $\Phi_1(f_*)\leq0$
and $\int\Delta Qf_*=\int\partial_{x_j}Qf_*=0$ for
$j=1,\cdots,N$.
$\int\Delta Qf_*=0$, however, yields that $\Phi_1(f_*)\geq0$  by step 1.
Therefore, we obtain that
\begin{align}\label{a12}
\Phi_1(f_*)=0
\end{align}
and that $f_*$ solves the following minimization problem
$$0=\frac{\int(L_+f_*)f_*
}{\|f_*\|_2}=\min_{f\in E\setminus\{0\}}\frac{\int(L_+f)f
}{\|f\|_2},$$ where
$E\equiv\{f\in H^1: \int\Delta Qf=\int\partial_{x_j}Qf=0, j=1,\cdots,N
\}$.
Hence, there exist some Lagrange multipliers $\lambda_k$, $k=0,1,\cdots,N$ such that
\begin{align}\label{a13}
L_+f_*=\lambda_0\Delta Q+\lambda_j\partial_{x_j}Q, \ \ j=1,\cdots,N.
\end{align}
By symmetry of $Q$, we get that $\int\partial_{x_j}Q\partial_{x_k}Q=0$ for $j\neq k$ and
$\int \partial_{x_j}Q\Delta Q=0$, which together with
 Proposition \ref{spectral}  imply that $$0=-\int f_*L_+(\partial_{x_j}Q)
=\int L_+f_*\partial_{x_j}Q=\lambda_j\int|\partial_{x_j}Q|^2,$$
showing that $\lambda_j=0$ for $j=1,\cdots,N$.
Thus, \begin{align}\label{a14}
L_+f_*=\lambda_0\Delta Q=\lambda_0(-Q^p+(1-s_c)Q.
\end{align}
Denote $\tilde Q=\frac2{p-1}Q+x\cdot Q$, then $\tilde Q=\frac \partial{\partial_\lambda}(Q_\lambda)|_{\lambda=1}$,
where $Q_\lambda\equiv\lambda^{\frac 2{p-1}}Q(\lambda x)$.
Differentiating the equality
$-\Delta Q_\lambda+\lambda^2(1-s_c)Q_\lambda-Q_\lambda^p=0$ with respect to $\lambda$ at $\lambda=1$,
we obtain that $ L_+\tilde Q=-2(1-s_c)Q$. Since $L_+Q=-(p-1)Q^p$,
we obtain that
\begin{align}\label{a15}
L_+(\frac{\lambda_0}{p-1}Q-\frac{\lambda_0}2\tilde Q)=\lambda_0(-Q^p+(1-s_c)Q).
\end{align}
In view of Proposition \ref{spectral}, \eqref{a14} and \eqref{a15} imply that
$f_*=\frac{\lambda_0}{p-1}Q-\frac{\lambda_0}2\tilde Q+\sum_{j=1}^N\mu_j\partial_{x_j}Q$
for some $\mu_j$.
Since $\int\tilde Q\partial_{x_j}Q=0$ and $\int f_*\partial_{x_j}Q=0$, we get that $\mu_j=0$ for $j=1,\cdots,N$.
Hence, $f_*=\frac{\lambda_0}{p-1}Q-\frac{\lambda_0}2\tilde Q=-\frac{\lambda_0}2(x\cdot\nabla Q)$.
By calculation, we obtain that $\Phi_1(f_*)=-\frac{\lambda_0^2}4\int\Delta Q(x\cdot\nabla Q)=-\frac{\lambda_0^2}8\int|\nabla Q|^2$,
which by \eqref{a12} implies that $\lambda_0=0$ and then $f_*=0$.
This contradicts $f_*\neq0$ obtained before.  We conclude that $\Phi_1(h_1)\geq c_1\|h_1\|_2^2$\
under the assumptions \eqref{orth} and \eqref{orth2}.
To complete the proof, it suffices to show that for some $c_2>0$,
$$\int Qh_2=0\ \ \Rightarrow\ \ \Phi_2(h_2)\geq c_2\|h_2\|_2^2.$$
The proof is similar as for $\Phi_1$ and we skip it.

Now we turn to show  the
coercivity of $\Phi$ on $G'_\perp$ also by two steps:

Firstly, we show that for any $h\in G'_\perp\setminus\{0\}$, $\Phi(h)>0$.
In fact, otherwise,  there exists $\tilde h\in H^1\setminus\{0\}$ such that
\begin{align}\label{a18}
\int\partial_{x_j}Q\tilde h_1=\int Q\tilde h_2=\int\mathcal Y_1\tilde h_2=\int\mathcal Y_2\tilde h_1=0,
\ \ \Phi(\tilde h)\leq0, \ \ j=1,\cdots,N.
\end{align}
By Proposition \ref{spectral}, $B(\partial_{x_j}Q,h)=B(iQ,h)=0$ for any $h\in H^1$.
Since, by \eqref{a18},  we also have that $B(\mathcal Y_+,\tilde h)=0$,
so we have that $\partial_{x_j}Q, iQ, \mathcal Y_+$ and $\tilde h$ are orthogonal in the bilinear symmetric form $B$.
Note that
$\Phi(iQ)=\Phi(\partial_{x_j}Q)=\Phi(\mathcal Y_+)=0$ and $\Phi(\tilde h)\leq0$, then
we get that
for any $h\in E\equiv span\{\partial_{x_j}Q, iQ, \mathcal Y_+, \tilde h, j=1,\cdots,N \}$, $\Phi(h)\leq0$.
Following the proof of \cite{holmer3}, we can claim that the dimension of the set $E$ is $N+3$.
Since we have known that $\Phi$ is  definite positive on $G_\perp$, which is a subspace of codimension $N+2$ of $H^1$,
then
$\Phi$ cannot be non-positive on $E$ with $dim E=N+3$. Thus we have got a contradiction, and
the proof of $\Phi(h)\leq0$ is complete.

The second step of the proof of  coercivity on $G'_\perp$ can be obtained similar to that
on $G_\perp$ by contradiction arguments and we omit the details.

\end{proof}

%We skip the main argument since it is similar to that in \cite{holmer3}except for the details.
%$\tilde Q=\frac2{p-1}Q+x\cdot\nabla Q$ we have $L_+\left(\frac{\lambda_0}{p-1}Q-\frac{\lambda_0}{2}\tilde Q\right)=\lambda_0\Delta Q$.

\begin{remark}\label{r2.1}
As a consequence of Proposition \ref{coercivity}, we claim that
\begin{align}\label{2.19}
\int(\Delta Q-(1-s_c)Q)\mathcal Y_1\neq0.
\end{align}
In fact, if otherwise  $\int(\Delta Q-(1-s_c)Q)\mathcal Y_1=0$,
then, by the equation \eqref{Q}, we have $\int L_+Q\mathcal Y_1=0$, which, by
\eqref{e_0}, implies that $\int Q\mathcal Y_2=0$. Thus, we obtain  $Q\in G_\perp^\prime$
and, from Proposition \ref{coercivity}, $\Phi(Q)>0$, which contradicts \eqref{phiq}.
\end{remark}

\section{Existence of spectral solutions}

We construct the solutions $Q^+$ and $Q^-$ of Theorem \ref{th2} in this section.

\begin{proposition}\label{UA}
Let $A\in\mathbb R$. If $t_0=t_0(A)>0$ is large enough, then there exists a radial solution
$U^A\in C^\infty([t_0,+\infty), H^\infty)$
of \eqref{1.1} such that
for any $b\in\mathbb R$ there exists $C>0$ such that
\begin{align}\label{3.1}
\|U^A(t)-e^{i(1-s_c)t}Q-Ae^{(i-e_0)t}\mathcal Y_+\|_{H^b}\leq Ce^{-2e_0t}.
\end{align}
\end{proposition}
\begin{remark}\label{label}
By \eqref{3.1},
\begin{align}\label{re32}
\|\nabla U^A(t)\|_2^2=\|\nabla Q\|^2_2
+2Ae^{-e_0t}\int(\nabla Q\cdot\nabla\mathcal Y_1+(1-s_c)Q \mathcal Y_1)+O(e^{-2e_0t}),
\end{align}
as $t\rightarrow+\infty$.
In view of  \eqref{2.19}, we may assume, without loss of generality, that $\nabla Q\cdot\nabla\mathcal Y_1+(1-s_c)Q \mathcal Y_1>0,$
and thus, $\|\nabla U^A(t)\|_2^2-\|\nabla Q\|^2_2$ has the sign of $A$ for large positive time.
\end{remark}
If we set
\begin{align}\label{Q^+}
Q^+(x,t)=e^{-i(1-s_c)t_0}U^{+1}(x,t+t_0),\ \ \ Q^-(x,t)=e^{-i(1-s_c)t_0}U^{-1}(x,t+t_0),
\end{align}
then we have got that $Q^\pm$ satisfy the statement in Theorem \ref{th2}
except for their behavior for the negative time, which we shall specify in Section 5 and Section 6.

\subsection{Approximate solutions}
First in this subsection, we restate the following proposition which is for the construction
of the approximate solutions $U_k^A$ of \eqref{1.1}.
\begin{proposition}\label{ZA}
Let $A\in\mathbb R$. There exists a sequence $\{\mathcal Z_j^A\}_{j\geq1}\subset \mathcal S$
such that $\mathcal Z_1^A=A\mathcal Y_+$ and if $k\geq1$ and
$\mathcal V_k^A\equiv \sum_{j=1}^ke^{-je_0t}\mathcal Z_j^A,$ then as $t\rightarrow+\infty$
\begin{align}\label{3.2}
\partial_t\mathcal V_k^A+\mathcal L\mathcal V_k^A=R(\mathcal V_k^A)+O(e^{-(k+1)e_0t})\ \ \ in\ \ \mathcal S.
\end{align}
\end{proposition}

\begin{remark}\label{r3.5}
Let $U_k^A\equiv e^{i(1-s_c)t}(Q+\mathcal V_k^A)$.
Then $U_k^A$ is an approximate solution of \eqref{1.1} which satisfies \eqref{3.1} for large $t$.
Indeed, as $t\rightarrow+\infty$, we have
$$i\partial_tU_k^A+\Delta U_k^A+|U_k^A|^{p-1}U_k^A=O(e^{-(k+1)e_0t})\ \ in\ \ \mathcal S.$$
\end{remark}
The proof of Proposition \ref{ZA} is almost the same as that in \cite{holmer3}, so we only sketch it now:

In fact, the proposition is proved by induction. Omitting the superscript $A$, we define first
$\mathcal Z_1=A\mathcal Y_+$ and $\mathcal V_1=e^{-e_0t}\mathcal Z_1$, which yields \eqref{3.2}
for $k=1$. Let $\mathcal Z_1,\cdots, \mathcal Z_k, k\geq1$ are known with the corresponding
$\mathcal V_k$ satisfying \eqref{3.2}. Expand the expression $R(\mathcal V_k)$ and by \eqref{3.2},
there exists $\mathcal U_{k+1}\in \mathcal S$ such that
$$\partial_t\mathcal V_k+\mathcal L\mathcal V_k=R(\mathcal V_k)+e^{-(k+1)e_0t}\mathcal U_k+O(e^{-(k+1)e_0t})\ \ in \ \ \mathcal S.$$
By Proposition \ref{spectral}, $(k+1)e_0$ is not in the spectrum of $\mathcal L$, so we can define $\mathcal Z_{k+1}=
-(\mathcal L-(k+1)e_0)^{-1}\mathcal U_{k+1}\in \mathcal S$ and $\mathcal V_{k+1}=\mathcal V_{k}+e^{-(k+1)e_0t}\mathcal Z_{k+1}$.
Thus, as $t\rightarrow+\infty$,
$$\partial_t\mathcal V_{k+1}+\mathcal L\mathcal V_{k+1}-R(\mathcal V_{k})
=R(\mathcal V_{k})-R(\mathcal V_{k+1})+O(e^{-(k+2)e_0t})\ \ in \ \ \mathcal S.$$
Since $\mathcal V_j=O(e^{-e_0t})$ in $\mathcal S$ for $j=k, k+1$,
and $\mathcal V_k-\mathcal V_{k+1}=O(e^{-(k+1)e_0t})$, we obtain then
$R(\mathcal V_{k})-R(\mathcal V_{k+1})=O(e^{-(k+2)e_0t})$ in $\mathcal S$, as
$t\rightarrow+\infty$. Thus, we have obtained \eqref{3.2} for $k+1$ and complete the proof.

%Proof of Proposition \ref{ZA}:

In the following subsections, we shall prove Proposition \ref{UA}.
\subsection{Construction of special solutions}

We construct a solution $U^A$ of \eqref{1.1} such that there exists $t_0\in \mathbb R$ satisfying
\begin{align}\label{3.5}
\forall b\in\mathbb R,\ \  \exists C>0:\ \ \forall t\geq t_0, k\in\mathbb N,\ \
\|U^A(t)-e^{i(1-s_c)t}(Q+\mathcal V_k^A(t))\|_{H^b}\leq Ce^{-2e_0t}
\end{align}
with $\mathcal V_k^A$  constructed in Proposition \ref{ZA}.
Note that \eqref{3.5} implies \eqref{3.1}, and that
if we have shown it for some $b_0$, it follows for $b\leq b_0$.
Thus, we only consider the case $b>N/2$, since then, it is well-known that the Sobolev space $H^b$
is a Banach algebra and we have the estimate $\|fg\|_{H^b}\leq C\|f\|_{H^b}\|g\|_{H^b}$
for any $f,g\in H^b$.
In order to do this, we write $$U^A=e^{i(1-s_c)t}(Q+h^A).$$
We are going to construct a solution of \eqref{linea}
$h^A\in C^0([t_k,+\infty),H^b)$ for $k$ and $t_k$ large such that
\begin{align}\label{3.6}
\|h^A(t)-\mathcal V_k^A(t)\|_{H^b}\leq Ce^{-(k+\frac12)e_0t}.
\end{align}
After that, we show by uniqueness argument that $h^A$ is independent of $b$ and $k$.
In the sequel, we omit the superscript $A$ for brevity.

Recall the equation
\eqref{linea2} of $h$ and
 define
\begin{align}\label{3.8}
\varepsilon_k(t)=i\partial_t\mathcal V_k+\Delta\mathcal V_k-(1-s_c)\mathcal V_k+S(\mathcal V_k)
\end{align}
for $k\in\mathbb N$. Then,
if we set $v\equiv h-\mathcal V_k$,
from \eqref{linea2} and \eqref{3.8}, we obtain that
\begin{align}\label{3.10}
i\partial_tv+\Delta v-(1-s_c)v=-S(\mathcal V_k+v)+S(\mathcal V_k)-\varepsilon_k.
\end{align}
Note that Proposition \ref{ZA} gives
\begin{align}\label{3.9}
\varepsilon_k(t)=O(e^{-(k+1)e_0t}).
\end{align}
We solve the corresponding integral equation
\begin{align}\label{integ}
v(t)=\mathcal M(v)(t),
\end{align}
where
$$\mathcal M(v)(t)\equiv-i\int_t^\infty e^{i(t-s)(\Delta-(1-s_c))}\Big(
S(\mathcal V_k(s)+v(s))-S(\mathcal V_k(s))+\varepsilon_k(s)
\Big)ds.$$

Note that \eqref{3.6} is equivalent to
$\|v(t)\|_{H^b}\leq Ce^{-(k+1/2)e_0t}$, for $t\geq t_k$.
Thus, we need show that $\mathcal M$ is a contraction on $B$, which is defined by
$$B=B(t_k,k,b)\equiv\{v\in E,\|v\|_E\leq1\},$$where
$$E=E(t_k,k,b)\equiv\{v\in C^0([t_k,+\infty),H^b),\|v\|_E\equiv\sup_{t\geq t_k}e^{(k+\frac12)e_0t}
\|v(t)\|_{H^b}<\infty\}.$$

Let $v\in B$.
Observe that for all $t\in\mathbb R$,  $e^{it(\Delta-(1-s_c))}$ is an isometry of $H^b$.
By definition of $S$ we have that
\begin{align}\label{3.12}
\|S(f)-S(g)\|_{H^b}\leq C\|f-g\|_{H^b}(1+\|f\|^{p-1}_{H^b}+\|g\|^{p-1}_{H^b}).
\end{align}
Then, for any $t\geq t_k$,
\begin{align}\label{3.13}
\|\mathcal M(v)(t)\|_{H^b}\leq C\int_t^\infty
\|v\|_{H^b}(1+\|\mathcal V_k(s)\|^{p-1}_{H^b}+\|v(s)\|^{p-1}_{H^b})ds+C_k\int_t^\infty
e^{-(k+1)e_0s}ds.
\end{align}
By the construction of $\mathcal V_k$, $\|\mathcal V_k(s)\|_{H^b}\leq C_ke^{-e_0s}$.
Moreover, since $v\in B$, $\|v(s)\|_{H^b}\leq Ce^{-(k+\frac12)e_0s}$.
Hence, for any $t\geq t_k$,
\begin{align}
\int_t^\infty
\|v\|_{H^b}(1+\|\mathcal V_k(s)\|^{p-1}_{H^b}+\|v(s)\|^{p-1}_{H^b})ds
&\leq C\int_t^\infty e^{-(k+\frac12)e_0s}+C_ke^{-(k+\frac12+p-1)e_0s}ds\\ \nonumber
&\leq C e^{-(k+\frac12)e_0t}\left(\frac1{(k+\frac12)e_0}+C_k e^{-(p-1)e_0t}
\right).
\end{align}
Therefore, $\mathcal M(v)\in E$ and by \eqref{3.13},
$$\|\mathcal M(v)\|_E\leq \frac C{(k+\frac12)e_0}+C_k e^{-\frac{e_0}{2}t_k}.$$
Choose $k$ large so that $\frac C{(k+\frac12)e_0}<\frac12$ and then take $t_k$ large such that
 $C_k e^{-\frac{e_0}{2}t_k}<\frac12$. Then $\mathcal M$ maps $B=B(t_k,k,b)$ to itself.
Similarly, we can also prove that $\mathcal M$ is a contraction on $B$.

We now show that $U^A$ is independent of $b$ and $k$.  By the preceding step, for $b_0=[\frac N2]+1$ there exist
$k_0$ and $t_0$ such that there exists a unique solution $U^A$ of \eqref{1.1}
satisfying $U^A\in C^0([t_0,\infty);H^{b_0})$ and for all $t\geq t_0$,
\begin{align}\label{3.14}
\|U^A(t)-e^{i(1-s_c)t}(Q+\mathcal V_{k_0}^A(t))\|_{H^{b_0}}\leq C e^{-(k_0+\frac12)e_0t}.
\end{align}
%By the equation \eqref{1.1}, we indeed show that $U^A\in C^\infty([t_0,\infty);H^{b_0})$ for any $b\in\mathbb R$.
Now, let $b_1>b_0$, if $k_1\geq k_0+1$ is large enough, there exist $t_1$ and $\tilde U^A\in C^0([t_1,\infty);H^{b_1})$
such that for all $t\geq t_0$,
\begin{align*}
\|\tilde U^A(t)-e^{i(1-s_c)t}(Q+\mathcal V_{k_1}^A(t))\|_{H^{b_1}}\leq C e^{-(k_1+\frac12)e_0t}.
\end{align*}
By the construction of $\mathcal V_k^A$,
$$\|\mathcal V_{k_1}^A-\mathcal V_{k_0}^A\|_{H^{b_1}}\leq C e^{-(k_0+1)e_0t}.$$
Then, we have that
\begin{align}\label{3.15}
\|\tilde U^A(t)-e^{i(1-s_c)t}(Q+\mathcal V_{k_0}^A(t))\|_{H^{b_1}}\leq  e^{-(k_1+\frac12)e_0t}+C e^{-(k_0+1)e_0t}
\leq Ce^{-(k_0+1)e_0t}.
\end{align}
In particular, $\tilde U^A$ satisfies \eqref{3.14} for large $t$. By uniqueness in the fixed point argument $\tilde U^A=U^A$, and
then, $U^A\in C^0([t_1,\infty);H^{b_1})$. By the persistence of regularity of \eqref{1.1},
 $U^A\in C^0([t_0,\infty);H^{b_1})$ and thus  $U^A\in C^0([t_0,\infty);H^b)$ for any $b\in\mathbb R$.
 By the equation \eqref{1.1}, we indeed show that $U^A\in C^\infty([t_0,\infty);H^{b})$ for any $b\in\mathbb R$.
 Note that \eqref{3.15} implies \eqref{3.5}, which conclude the proof of Proposition \ref{UA}.
\ \ \ \ \ \ \ \ \ \ \ \ \ \ \ \
  \ \ \ \ \ \ \ \ \ \ \ \ \ \ \ \ \ \ \ \ \ \ \ \ \ \ \ \ \ \ \ \ \ \ \ \ \ \ \ \ \ \ \ \ \ \ \ \ \ \ \ \ $\Box$

%By the preceding step with $b=N$, there exists $k_0$ and $t_0$ such that there exists a unique solution $U^A$ of \eqref{1.1} satisfying
%$U^A\in C^0([t_k,+\infty),H^N)$ and \begin{align}\label{3.14} \|U^A(t)-e^{i(1-s_c)t}(Q+\mathcal V_{k_0}^A(t))\|_{H^N}\leq Ce^{-(k_0+1/2)e_0t}.
%\end{align}
%后面证明参考merle能量临界的证明

\section{modulation of threshold solutions}

For $u\in H^1$, we define
\begin{align}\label{delta}
\delta(u)=\Big |\int |\nabla Q|^2-\int |\nabla u|^2\Big |.
\end{align}
The variational characterization of $Q$ (Proposition \ref{charaQ}) shows that
if\footnote{ Note that, by the same argument in the proof of Proposition
\ref{p21'}, any solution satisfying \eqref{1.3} can be rescaled to the one satisfying \eqref{4.1}.}
\begin{align}\label{4.1}
M(u)=M(Q),\ \ \ E(u)=E(Q),
\end{align}
and $\delta(u)$ is small enough, then there exists $\tilde\theta$ and $\tilde x$
such that $u_{\tilde\theta,\tilde x}\equiv e^{-i\tilde\theta}u(\cdot+\tilde x)=Q+\tilde u$
with $\|\tilde u\|_{H^1}\leq\tilde\varepsilon(\delta(u))$,
where $\tilde\varepsilon(\delta(u))\rightarrow0$ as $\delta\rightarrow0.$
Now for the solution $u$ of equation \eqref{1.1} with small gradient variant away from $Q$,
we aim to introduce a choice of modulation parameters $\sigma$ and $X$ for which
the quantity $\delta(u)$ controls linearly $\|u_{\sigma,X}-Q\|_{\dot H^1}$
and other relevant parameters of the problem. The choice of parameters
is made through two orthogonality conditions given by the two groups of transformations
$u\mapsto e^{-i\sigma}u, \sigma\in\mathbb R$ and  $u\mapsto u(\cdot+X), X\in\mathbb R^N$.

We first give a useful lemma as follows.
\begin{lemma}\label{l41}
There exist $\delta_0>0$ and a positive function $\varepsilon(\delta)$ defined for
$0<\delta\leq\delta_0$, which tends to $0$ as $\delta\rightarrow0$
such that for all $u\in H^1$ satisfying \eqref{4.1} and $\delta(u)<\delta_0$,
there exists a couple $(\sigma,X)\in \mathbb R\times\mathbb R^N$ such that
$v=e^{-i\sigma}u(\cdot+X)$ satisfies
\begin{align}\label{4.2}
\|v-Q\|_{H^1}\leq\varepsilon(\delta),
\end{align}
\begin{align}\label{4.3}
Im \int Qv=0,\ \ \ Re\int\partial_{x_k}Qv=0,\ \ k=1,\cdots,N.
\end{align}
The parameters $\sigma$ and $X$ are unique in $\mathbb R/2\pi\mathbb Z\times \mathbb R^N$, and
the mapping $u\mapsto(\sigma,X)$ is $C^1$.

\end{lemma}

\begin{proof}
Consider the functionals on $\mathbb R\times\mathbb R^N\times H^1$:
$$J_0: (\sigma, X, u)\mapsto Im\int e^{-i\sigma}u(x+X)Q,\ \ J_k:
(\sigma, X, u)\mapsto Re\int e^{-i\sigma}u(x+X)\partial_kQ,\ \ k=1,\cdots, N.$$
Thus, the orthogonality conditions \eqref{4.3} are equivalent to the conditions
$J_j(\sigma, X, u)=0$, $j=0,\cdots,N$.
Note that $J_j(0,0,Q)=0$ for $j=0,\cdots,N$. By direct calculation, one can check that
for $j=0,\cdots,N$ and $k=1,\cdots,N$,
$\left(\frac{\partial J_j}{\partial \sigma},\frac{\partial J_j}{\partial X_k}
\right)$ is invertible at $(0,0,Q)$. By the Implicit Function Theorem, there exist $
\epsilon_0,\eta_0>0$ such that for $u\in H^1$ satisfying
$\|u-Q\|_{H^1}<\epsilon_0$, there exists $(\sigma, X)\in\mathbb R\times\mathbb R^N$ with
$|\sigma|+|X|\leq\eta_0$ such that $J_j(\sigma, X,Q)=0$.
Now for $u\in H^1$ satisfying \eqref{4.1} and $\delta(u)<\delta_0$,
by Proposition \ref{charaQ}, we can choose $\tilde\theta$ and $\tilde X$
such that $e^{-i\tilde\theta}u(\cdot+\tilde X)$ is close to $Q$ in $H^1$,  and so, as argued above,  get
$(\sigma, X)\in\mathbb R\times\mathbb R^N$ required in the lemma.
Also by the Implicit Function Theorem, we can show the uniqueness of $(\sigma, X)$ and the regularity of the mapping
$u\mapsto (\sigma, X)$, concluding the proof.

\end{proof}

Let $u$ be a solution of \eqref{1.1} satisfying \eqref{4.1}. For convenience, we write
$\delta(t)\equiv\delta(u(t))$ and set
$D_{\delta_0}\equiv\{t: \delta(t)<\delta_0\}.$
By Lemma \ref{l41}, we can define functions $\sigma(t),X(t)\in C^1$ on $D_{\delta_0}$.
Using the modulation theory to do some perturbative analysis, we write
\begin{align}\label{4.4}
e^{-i\theta(t)-i(1-s_c)t}u(t,x+X(t))=(1+\alpha(t))Q(x)+h(t,x),
\end{align}
with
$$\alpha(t)=Re\frac{e^{-i\theta(t)-i(1-s_c)t}\int \nabla u(t,x+X(t))\cdot\nabla Q(x)
}{\int|\nabla Q|^2
}-1.$$
In fact, we choose  $\alpha$ like this such that $h$ satisfies the orthogonality condition \eqref{orth2}.

\begin{lemma}\label{l42}
Let the solution $u$ of \eqref{1.1} satisfy \eqref{4.1}.
Taking  $\delta_0$ small if necessary,
the following estimate hold for $t\in D_{\delta_0}$:
\begin{align}\label{4.5}
|\alpha(t)|\approx\Big|\int Qh_1(t)\Big|\approx\|h(t)\|_{H^1}\approx\delta(t).
\end{align}
\end{lemma}

\begin{proof}
Let $\tilde\delta(t)\equiv|\alpha(t)|+\delta(t)+\|h(t)\|_{H^1}$. By  Lemma \ref{l41}, we know
that $\tilde\delta(t)$ is small when $\delta(t)$ is small. From the equalities
$M(Q+\alpha Q+h)=M(u)=M(Q)$
we obtain $\int|\alpha Q+h|^2+2\alpha\int Q^2+2\int Qh_1,$
which implies then
\begin{align}\label{4.6}
|\alpha(t)|=\frac 1{M(Q)}\Big|\int Qh_1(t)\Big|+O(\tilde\delta^2).
\end{align}

By the orthogonality condition \eqref{orth2}, we get
$$\delta(t)=\Big|\int |\nabla(Q+\alpha Q+h)|^2-\int|\nabla Q|^2\Big|
=\Big|(2\alpha+\alpha^2)\int |\nabla Q|^2+\int|\nabla h|^2\Big|,$$
which implies
\begin{align}\label{4.7}
|\alpha(t)|=\frac 1{2\|\nabla Q\|_2^2}\delta+O(\tilde\delta^2).
\end{align}
The orthogonality condition $\int\nabla Q\cdot\nabla h_1=0$
together with the equation \eqref{Q} implies that
$\int Q^{p}h_1=(1-s_c)\int Qh_1.$
Thus, $B(Q,h)=-\frac 12(p-1)(1-s_c)\int Qh_1=-(1-\frac{(N-2)(p-1)}4)\int Qh_1.$
This combined with \eqref{Phi2} gives
$$\Big|\alpha^2\Phi(Q)+\Phi(h)-2\alpha
\int Qh_1\Big|
=|\Phi(\alpha Q+h)|=O(\alpha^3+\|h\|_{H^1}^3).$$
So
\begin{align}\label{4.8}
\Phi(h)=
\alpha^2|\Phi(Q)|+2\alpha\int Qh_1
+O(\alpha^3+\|h\|_{H^1}^3).
\end{align}
On the other hand,
by Proposition \ref{coercivity} and \eqref{Phi2},
$\Phi(h)\approx\|h\|^2_{H^1},$
which together with \eqref{4.8} implies that
\begin{align}\label{4.9}
\|h\|_{H^1}=O(|\alpha|+\Big|\int Qh_1\Big|+\tilde\delta^{3/2}).
\end{align}
Now, \eqref{4.6} combined with \eqref{4.9} gives
$\|h\|_{H^1}=O(|\alpha|+\tilde\delta^{3/2})$.
Thus, by the definition of $\tilde\delta$,  \eqref{4.6},\eqref{4.7} and \eqref{4.9} yields \eqref{4.5} immediately.
\end{proof}

Using Lemma \ref{l41} and Lemma \ref{l42} we have the following two lemmas.
\begin{lemma}\label{l43}
Under the assumption of Lemma \ref{l42}, taking smaller $\delta_0$ if necessary, we have for
$t\in D_{\delta_0}$
\begin{align}\label{4.10}
|\alpha'|+|X'|+|\theta'|=O(\delta).
\end{align}
\end{lemma}
%This lemma can be shown by the same argument as Lemma \ref{l42}
\begin{proof}
Let $\delta^*=\delta(t)+|\alpha'(t)|+|X'(t)|+|\theta'(t)|$.
By \eqref{4.4} and Lemma \ref{l42},
the equation \eqref{1.1} can be rewritten as
\begin{align}\label{4.11}
i\partial_th+\Delta h+i\alpha'Q-\theta'Q-iX'\cdot\nabla Q=O(\delta+\delta\delta^*).
\end{align}
Firstly, multiplying \eqref{4.11} by $Q$ and integrating  the real part on $\mathbb R^N$, we obtain
from \eqref{orth2} that
$|\theta'|=O(\delta+\delta\delta^*)$.
Then by  multiplying \eqref{4.11} by $\partial_{x_j}Q, j=1,\cdots,N$ and integrating  the imaginary part, we obtain
from Lemma \ref{l42} and $\int\Delta h\partial_{x_j}Q=O(\delta)$ that
$|X_j'|=O(\delta+\delta\delta^*)$. Similarly,  by
multiplying \eqref{4.11} by $\Delta Q$ and integrating  the imaginary part, we obtain that
$|\alpha'|=O(\delta+\delta\delta^*)$.
As a consequence, we  obtain that $\delta^*=O(\delta+\delta\delta^*)$
which  concludes our proof by choosing $\delta_0$ small enough.

\end{proof}

\begin{lemma}\label{l44}
Let $u$ be a solution of \eqref{1.1} satisfying \eqref{4.1}.
Assume that $u$ is defined on $[0,+\infty)$
and that there exist $c,C$ such that for any $t\geq 0$,
\begin{align}\label{4.16}
\int_t^\infty\delta(s)ds\leq Ce^{-ct}.
\end{align}
Then there exist $\theta_0\in\mathbb R, x_0\in\mathbb R^N$ and $c,C>0$
such that $$\|u-e^{i(1-s_c)t+i\theta_0}Q(\cdot-x_0)\|_{H^1}\leq Ce^{-ct}.$$
\end{lemma}

\begin{proof}
We first announced
\begin{align}\label{4.22}
\lim_{t\rightarrow+\infty}\delta(t)=0.
\end{align}
In fact, if not, by \eqref{4.16}, there exist two increasing sequences $t_n$ and $t_n'$ such that
$t_n<t_n'$, $\delta(t_n)\rightarrow0$, $\delta(t_n')=\epsilon_1$ for some $0<\epsilon_1<\delta_0$,
and for any $t\in(t_n,t_n')$, there holds that $0<\delta(t)<\epsilon_1$.
On $[t_n,t_n']$, $\alpha(t)$ is well-defined. By Lemma \ref{l43}, $|\alpha'(t)|=O(\delta(t))$,
so by \eqref{4.16}, $\int_{t_n}^{t_n'}|\alpha'(t)|dt\leq Ce^{-ct_n}$.
Hence,
\begin{align}\label{4.21}
\lim_{n\rightarrow+\infty}|\alpha(t_n)-\alpha(t_n')|=0.
\end{align}
By Lemma \ref{l42}, we have $|\alpha(t)|\approx\delta(t)$.
Then,  the assumption $\delta(t_n)\rightarrow0$ yields that $|\alpha(t_n)|\rightarrow0$,
which, by \eqref{4.21}, implies $|\alpha(t_n')|\rightarrow0$ showing a contradiction with the assumption.
We have shown the claim \eqref{4.22}.

By \eqref{4.22},
 Lemma \ref{l42} and Lemma \ref{l43}, we obtain that
\begin{align}\label{4.23}
\delta(t)\approx\|h(t)\|_{\dot H^1}\approx|\alpha(t)|=\left|-\int_{t}^\infty\alpha'(s)ds\right|
\leq C\int_{t}^\infty|\alpha'(s)|ds\leq
\int_t^\infty\delta(s)ds\leq
 Ce^{-ct}.
 \end{align}
Furthermore, since by Lemma \ref{l43}, $|X'(t)|+|\theta'(t)|=O(\delta(t))\leq Ce^{-ct}$,
then there exist $X_\infty$ and $\theta_\infty$ such that
\begin{align}\label{4.23'}
|X(t)-X_\infty|+|\theta(t)-\theta_\infty|\leq Ce^{-ct}.
\end{align}
In view of the decomposition \eqref{4.4} of $u$,
\eqref{4.23} and  \eqref{4.23'} conclude the proof of Lemma \ref{l44} immediately.
\end{proof}

\section{convergence to $Q$ in the case $\|\nabla u_0\|_2\| u_0\|_2>\|\nabla Q\|_2\|Q\|_2$}
The goal of this section is to prove the following proposition:
\begin{proposition}\label{p51}
Consider a solution $u$ of \eqref{1.1} such that
\begin{align}\label{5.1}
E(u)=E(Q), \ \ \ M(u)=M(Q),
\end{align}
\begin{align}\label{5.2}
\|\nabla u_0\|_2>\|\nabla Q\|_2,
\end{align}
which is globally defined for positive times.
Assume furthermore that $u_0$ is either of finite variance, i.e.,
\begin{align}\label{5.3}
\int|x|^2|u_0|^2<+\infty,
\end{align}
or $u_0$ is radial. Then there exist $\theta_0\in\mathbb R, x_0\in\mathbb R^N$ and $c,C>0$
such that $$\|u-e^{i(1-s_c)t+i\theta_0}Q(\cdot-x_0)\|_{H^1}\leq Ce^{-ct}.$$
Moreover, the negative time of existence of $u$ is finite.
\end{proposition}

Note that Proposition \ref{p51} implies that the radial solution $Q^+$ constructed by \eqref{Q^+}
has finite negative time of existence.

\subsection{Finite variance solutions}
Proposition \ref{p51} in this case follows  from the following lemma.
\begin{lemma}\label{l53}
Let $u$ be a solution of \eqref{1.1} satisfying
\eqref{5.1}, \eqref{5.2}, \eqref{5.3} and
\begin{align}\label{5.4}
T_+(u_0)=+\infty.
\end{align}
Then for all $t$ in the interval of existence of $u$,
\begin{align}\label{5.5}
Im\int x\cdot\nabla u(x,t)\bar u(x,t)dx>0,
\end{align}
and there exist $c,C>0$ such that for any $t\geq0$,
\begin{align}\label{5.6}
\int_t^\infty\delta(s)ds\leq Ce^{-ct}.
\end{align}

\end{lemma}

Before proving this lemma, we first show how to use it to prove Proposition \ref{p51}.
Assuming  that $u$ is globally defined for negative times,  we
consider $v(x,t)=\bar u(x,-t)$. Thus, $v$ is a solution of \eqref{1.1}
satisfying the assumptions of Lemma \ref{l53}.
Applying \eqref{5.5} to $v$ for all $t$ in the domain of the existence of $u$, we get
$$0<Im\int x\cdot\nabla v(x,-t)\bar v(x,-t)dx=-Im\int x\cdot\nabla u(x,t)\bar u(x,t)dx,$$
which contradicts \eqref{5.5}.
Hence, the negative time of existence of $u$ is finite.
The other assertion of Proposition \ref{p51} follows from \eqref{5.6} and Lemma \ref{l44}.

Proof of Lemma \ref{l53}:

We set $y(t)\equiv\int|x|^2|u(x,t)|^2$. By calculation, we have that
$y'(t)=4Im\int x\cdot\nabla u\bar u$
and $$y''(t)=4N(p-1)E(u)-\left(2N(p-1)-8\right)\|\nabla u\|_{2}^{2}=4N(p-1)E(Q)-\left(2N(p-1)-8\right)\|\nabla u\|_{2}^{2}.$$
By \eqref{2.2}, we get that
\begin{align}\label{y''}
y''(t)=\left(2N(p-1)-8\right)\left(\|\nabla Q\|^2_2-\|\nabla u\|^2_2\right)=-\left(2N(p-1)-8\right)\delta(t)<0.
\end{align}
We show \eqref{5.5}, which is equivalent to $y'(t)>0$, by contradiction.
If it does not hold,  there exists some $t_1$ such that $y'(t_1)\geq0$.
Since by \eqref{y''}, $y''<0$,  then for $t_0>t_1$,
$$y'(t)\leq y'(t_0)<0,\ \ \ \forall\ \ t\geq t_0.$$
Since $T_+(u_0)=+\infty$, we obtain that $y(t)<0$ for large $t$, which is a contradiction and \eqref{5.5} must hold.

We next claim that
\begin{align}\label{5.8'}
(y'(t))^2\leq Cy(t)(y''(t))^2.
\end{align}
In fact, this claim follows from \eqref{y''} and the following lemma:
\begin{lemma}\label{claim}
Let $\phi\in C^1(\mathbb R^N)$ and $f\in H^1(\mathbb R^N)$. Assume that
$\int|f|^2|\nabla\phi|^2<\infty$ and $\|f\|_2=\|Q\|_2, E(f)=E(Q)$.
Then $$\left|Im\int(\nabla\phi\cdot\nabla f)\bar f
\right|^2\leq C\delta^2(f)\int|\nabla\phi|^2|f|^2.$$
\end{lemma}
This lemma
was shown in \cite{holmer3} for $N=3$.  Since for the general case, it is just an easy extension, we omit the proof.
Taking $\phi(x)=|x|^2$ in Lemma \ref{claim}, we get \eqref{5.8'}.

Now, for all $t$ in the interval of existence of $u$, we have that
$y'(t)>0$ and $y''(t)<0$ and thus,
\begin{align}\label{5.14}
\frac{y'(t)}{\sqrt{y(t)}}\leq -Cy''(t).
\end{align}
Integrating \eqref{5.14} on $[0,t]$, we get that $$\sqrt{y(t)}-\sqrt{y(0)}\leq -C(y'(t)-y'(0))\leq Cy'(0),$$
which shows that $y(t)$ is bounded for $t\geq0$.
Thus \eqref{5.14} gives in turn that  $y'(t)\leq -Cy''(t)$, which implies  then
$y'(t)\leq Ce^{-ct}.$
Since  $y'(t)=-\int_t^\infty y''(s)dx=\left(2N(p-1)-8\right)\int_t^\infty \delta(s)ds$,
then we obtain \eqref{5.6}, concluding  the proof of  Lemma \ref{l53}.
\ \ \  \ \ \ \ \ \ \ \ \ \ \ \ \ \ \ \ \ \ \ \ \ \ \ \ \ \ \ \ \ \ \ \ \ \ \ \ \ \ \ \ \ \ \ \ \ $\Box$

\subsection{Radial solutions}

For the radial solution $u$ of \eqref{1.1} that satisfies \eqref{5.1}, \eqref{5.2}
and is globally defined for positive time, we show in this subsection that $u$
has finite variance and finish the proof of Proposition \ref{p51} from the finite-variance case obtained above.

Let $\varphi$ be a radial function such that
$0\leq\varphi(r), \varphi''(r)\leq2$ and that
$\varphi(r)=r^2$ for $0\leq r\leq1$ while $\varphi(r)\equiv0$ for $r\geq2$.
Consider the localized variance $y_R(t)=\int R^2\varphi(\frac xR)|u(x,t)|^2dx$.
By \eqref{5.1}, we compute that
$$4N(p-1)E(u)-\left(2N(p-1)-8\right)\|\nabla u\|_{2}^{2}=
\left(2N(p-1)-8\right)\left(\|\nabla Q\|^2_2-\|\nabla u\|^2_2\right).$$
Since $u$ is radial, by explicit calculation, we obtain
\begin{align}\label{5.15}
y'_R(t)=2RIm\int\bar u\nabla\varphi (\frac xR)\cdot\nabla u,
\end{align}
and
\begin{align}\label{5.16}
&y''_R=4\sum_{j,k}Re\int\partial_k\partial_j\varphi (\frac xR)\partial_ku\partial_j\bar u
-\frac1{R^2}\int\Delta^2\varphi (\frac xR)|u|^2-\frac{2(p-1)}{p+1}\int\Delta\varphi (\frac xR)|u|^{p+1}\\ \nonumber
=&\left(2N(p-1)-8\right)\left(\|\nabla Q\|^2_2-\|\nabla u\|^2_2\right)+A_R(u)
=-\left(2N(p-1)-8\right)\delta(t)+A_R(u),
\end{align}
where
\begin{align}\label{5.17}
A_R(u(t))
=&4\sum_{j\neq k}\int\partial_j\partial_k\varphi(\frac xR)\partial_ju\partial_k\bar u
+4\sum_j\int\left(\partial^2_{x_j^2}\varphi(\frac xR)-2
\right)|\partial_ju|^2\\ \nonumber &
-\frac 1{R^2}\int\Delta^2\varphi(\frac xR)|u|^2
-\frac{2(p-1)}{p+1}\int\left(\Delta\varphi(\frac xR)-2N
\right)|u|^{p+1}\\ \nonumber
=&4\int\left(\varphi''(\frac xR)-2
\right)|\nabla u|^2
-\frac 1{R^2}\int\Delta^2\varphi(\frac xR)|u|^2
-\frac{2(p-1)}{p+1}\int\left(\Delta\varphi(\frac xR)-2N
\right)|u|^{p+1}.
\end{align}
We now claim that there exists $R_0>0$ such that for any $R>R_0$,
\begin{align}\label{yR''}
y''_R(t)\leq-\left(N(p-1)-4\right)\delta(t).
\end{align}
By \eqref{5.17}, we need to show that there exists $R_0>0$ such that for any $R>R_0$,
 $ A_R(u(t))\leq\left(N(p-1)-4\right)\delta(t)$.
 In fact, we first note that, for the standing-wave solution $e^{i(1-s_c)t}Q$ of \eqref{1.1}, the corresponding
 $y_R(t)$ is a constant and the $\delta(t)$ is identically zero, which imply that $A_R(e^{i(1-s_c)t}Q)=0$.
Now using the parameter $\delta_0$ as in section 4, we will show the claim \eqref{yR''} in two cases.

 Firstly, we assume that $t\in D_{\delta_1}$, where $\delta_1<\delta_0$ is to be chosen later.
If we denote $v\equiv \alpha Q+h$, we get from Lemma \ref{l42} that
 $$u(t)=e^{i(1-s_c)t}(Q+v(t)),\ \ \ \ \ \|v(t)\|_{H^1}\leq C\delta(t).$$
 Noting that $\varphi''(\frac xR)-2=\Delta^2\varphi(\frac xR)=\Delta\varphi(\frac xR)-2N=0$
 for $|x|\leq R$, we obtain that
 $$|A_R(u(t))|=|A_R(Q+v)-A_R(Q)|\leq C\int _{|x|\geq R}\Big(Q^p|v|+|v|^{p+1}+|\nabla Q||\nabla v|
 +|\nabla v|^2+Q|v|+|v|^2\Big).$$
 By the exponential decay of $Q$ at infinity, we get that  for $R>R_1>0$ large and $\delta_1$ sufficiently small,
 $$|A_R(u(t))|\leq C\Big(e^{-cR}\delta(t)+\delta(t)^2+\delta(t)^{p+1}\Big)\leq \left(N(p-1)-4\right)\delta(t).$$
 So  \eqref{yR''} holds for $R>R_1$ and $t\in D_{\delta_1}$.

 Next, we fix such a  $\delta_1$ and assume that $\delta(t)\geq\delta_1$. By our assumption on $\varphi$, we know that
 $\int\left(\varphi''(\frac xR)-2
\right)|\nabla u|^2\leq0$. It suffices to bound the other two terms now.
Since if $R\geq R_2=\sqrt{\frac{CM(Q)}{\delta_1}}$,
\begin{align}\label{5.20}
\frac 1{R^2}\int\Delta^2\varphi(\frac xR)|u|^2\leq \frac CM(Q)\leq\delta_1\leq \left(\frac{N(p-1)}{2}-2\right)\delta(t).
\end{align}
On the other hand,
from the Radial Gagliardo-Nirenberb inequality:
\begin{lemma}\label{raphael}\cite{raphael}
 For all~$\delta>0$, there exists a constant $C_{\delta}>0$ such that for all $u\in \dot{H}^{s_{c}}$
with radial symmetry, and for all $R>0,$ we have
$$\int_{|x|\geq R}|u|^{p+1}dx\leq\delta\int_{|x|\geq R}|\nabla u|^{2}dx
+\frac{C_{\delta}}{R^{2(1-s_{c})}}\left[\left(\rho(u,R)\right)^{\frac{2(p+3)}{5-p}}+\left(\rho(u,R)\right)^{\frac{p+1}{2}}\right],$$
where ~$\rho(u,R)=\sup_{R'\geq R}\frac{1}{(R')^{2s_{c}}}\int_{R'\leq|x|\leq2R'}|u|^{2}dx.$~
\end{lemma}
We have for all~$\epsilon>0$,~there exists a constant~$C_{\epsilon}>0$~and~$C_{Q}>0$~such that for all~$u\in \dot{H}^{s_{c}}$~
with radial symmetry and~$M(u)=M(Q)$~ and for all $R>0,$
\begin{align*}
\int_{|x|\geq R}|u|^{p+1}dx\leq\epsilon\int_{|x|\geq R}|\nabla u|^{2}dx
+\frac{C_{\epsilon}C_{Q}}{R^\beta},
\end{align*}
where $\beta=\min\{2+\frac{2s_c(3p+1)}{5-p},2+s_c(p-3)\}>0$.
Thus, for $\epsilon$ small  and $R>R_3$ large enough,
\begin{align}\label{radialsobolev}
C\int_{|x|\geq R}|u|^{p+1}dx\leq\epsilon(\delta(t)+\|\nabla Q\|_2^2)
+\frac{C_{\epsilon}}{R^\beta}\leq\epsilon C_{\delta_1}\delta(t)+\frac{C_{\epsilon}}{R^\beta}\leq \left(\frac{N(p-1)}{2}-2\right)\delta(t).
\end{align}
By \eqref{5.20} and \eqref{radialsobolev}, we get the claim \eqref{yR''} in the case $\delta(t)\geq\delta_1$ also.

Next, we claim that $y'_R(t)>0$ for all $t$ in the interval of existence of $u$.
In fact, if not, since $y''_R(t)<0$ by \eqref{yR''}, there must  exists $t_1,\epsilon>0$ such that
for $t\geq t_1$, $y'_R(t)<-\epsilon$, which contradicts the fact that $y_R$ is positive and $u$
is globally defined for positive time. Thus we conclude the claim.

Since $y'_R$ is positive and decreasing, it must have finite limit as $t\rightarrow+\infty.$
Since then the integral $\int_0^\infty y_R''(t)dt<\infty$ converges, this
 combined with \eqref{yR''} implies that $\int_0^\infty\delta(s)ds<\infty$.
Thus, there exists a subsequence $t_n\rightarrow+\infty$ such that $\delta(t_n)\rightarrow0$.
By Proposition \ref{charaQ}, there exists $\theta_0\in\mathbb R$ such that $u(t_n)
\rightarrow e^{i\theta_0}Q$ in $H^1$ up to a subsequence and  translation.
Since $y'_R(t)>0$, i.e.,  $y_R(t)$ is increasing, thus
$$y_R(0)=\int R^2\varphi(\frac xR)|u_0|^2\leq\int R^2\varphi(\frac xR)|u(t_n)|^2\leq\int R^2\varphi(\frac xR)|Q|^2.$$
Letting $R\rightarrow+\infty$, we obtain then
$\int|x|^2|u_0|^2<\infty,$
which turn the radial case to the finite-variance one and, by the argument in Subsection 5.1, we
have proved   Proposition \ref{p51}. \ \ \ \ \ \ \ \ \ \ \ \ \ \ \ \ \ \ \ \ \
 \ \ \ \ \ \ \ \ \ \ \ \ \ \ \ \ \ \ \ \ \ \ \ \ \ \ \ \ \ \ \ \ \ \ \ \ \ \ \ \ \ \ \ \ \ \
  \ \ \ \ \ \ \ \ \ \ \ \ \ \ \ \ \ \ \ \ \ \ \ \ \ \ \ \ \
   \ \ \ \ \ \ \ \ \ \ \ \ \ \ \ \ \ \ \ \ \ \ \ \ \ \ \ \ \ \ \ \ \ \ \ $\Box$

\section{convergence to $Q$ in the case $\|\nabla u_0\|_2\| u_0\|_2<\|\nabla Q\|_2\|Q\|_2$}
In this section we are to prove the following proposition and then finish the proof of Theorem \ref{th2}.
\begin{proposition}\label{p61}
Consider a solution $u$ of \eqref{1.1} such that
\begin{align}\label{6.1}
E(u)=E(Q), \ \ \ M(u)=M(Q),\ \
\|\nabla u_0\|_2<\|\nabla Q\|_2,
\end{align}
which does not scatter for positive times.
Then there exist $\theta_0\in\mathbb R, x_0\in\mathbb R^N$ and $c,C>0$
such that $$\|u-e^{it+i\theta_0}Q(\cdot-x_0)\|_{H^1}\leq Ce^{-ct}.$$
\end{proposition}
In subsection 6.1, we show that a solution satisfying \eqref{6.1} is compact
in $H^1$ up to a translation $x(t)$ in space. This is a consequence, through
the profile decomposition initially introduced by Keraani \cite{keraani}, of the scattering of subcritical solution of \eqref{1.1}
shown in \cite{yuan}. Then in subsection 6.2, it is shown, by a local virial identity, that the
parameter $\delta(t)$ converges to $0$ in mean. We conclude in subsection 6.3 the proof of Proposition \ref{p61}
using the results obtained above. Finally, in the last subsection 6.4, we are dedicated to the behavior of $Q^-$
for negative times, concluding the proof of Theorem \ref{th2}.

\subsection{Compactness properties}

\begin{lemma}\label{l62}
Let $u$ be a solution of \eqref{1.1} satisfying the assumptions of Proposition \ref{p61}.
Then there exists a continuous function $x(t)$ such that
\begin{align}\label{6.2}
K\equiv\{u(x+x(t),t),t\in[0,\infty)\}
\end{align}
has a compact closure in $H^1$.
\end{lemma}
We sketch the proof similar to that in \cite{G1}:
\begin{proof}
It suffices to show that for every sequence $\tau_n\geq0$, there exists a subsequence $x_n$
such that $u(x+x_n,\tau)$ has a limit in $H^1$.

We recall the profile decomposition discussed in \cite{G1}.
There exist $\psi^j\in H^1$ and sequences $x_n^j, t_n^j$ such that
\begin{align}\label{pd}
u(x,\tau_n)=\sum_{j=1}^{M}e^{-it_{n}^{j}\Delta}\psi^{j}(x-x_{n}^{j})+W_{n}^{M}(x),
\ \ \ \lim_{M\rightarrow+\infty}[\lim_{n\rightarrow+\infty}\|e^{it\Delta}W_{n}^{M}\|_{S(\dot{H}^{s_{c}})}]=0,
\end{align}
\begin{align}\label{pd2}
\lim_{n\rightarrow+\infty}(
|t_{n}^{j}-t_{n}^{k}|+|x_{n}^{j}-x_{n}^{k}|)=+\infty.
\end{align}
For fixed $M$ and any $0\leq s\leq1$, we have the asymptotic
Pythagorean expansion:
\begin{equation}\label{Hsexpan}
\|\phi_{n}\|_{\dot{H}^{s}}^{2}=\sum_{j=1}^{M}\|\psi^{j}\|_{\dot{H}^{s}}^{2}+\|W_{n}^{M}\|_{\dot{H}^{s}}^{2}+o_{n}(1),
\end{equation}
and the energy ~Pythagorean decomposition
\begin{equation}\label{Eexpan}
E(\phi_{n})=\sum_{j=1}^{M}E(e^{-it_{n}^{j}\Delta}\psi^{j})+E(W_{n}^{M})+o_{n}(1).
\end{equation}
We now show that there is exactly one nonzero profile.
On the one hand, if for all $j, \psi^j=0$, then $u$ must scatter by the small data theory
(Proposition \ref{sd}) and we get a contradiction.

On the other hand, if at least two profiles are nonzero,
then by the Pythagorean expansion \eqref{Hsexpan}, %properties of the profile decomposition,%(see \cite{G1}),
there exists $\epsilon>0$ such that for all $j$,
\begin{align}\label{6.5}
\|\psi^{j}\|^{\frac{1-s_c}{s_c}}_2\|\nabla\psi^j\|_2\leq\|Q\|^{\frac{1-s_c}{s_c}}_2\|\nabla Q\|_2-\epsilon,
\end{align}
which,  by the Gagliardo-Nirenberg inequality \eqref{2.1} and \eqref{2.3} , implies that
$E(e^{-it_{n}^{j}\Delta}\psi^j)>0.$
Thus, by the Pythagorean expansion \eqref{Eexpan}, we obtain also
\begin{align}\label{6.5'}
M(\psi^j)^{\frac{1-s_c}{s_c}}E(e^{-it_{n}^{j}\Delta}\psi^j)\leq M(Q)^{\frac{1-s_c}{s_c}}E(Q)-\epsilon.
\end{align}
 By the existence of wave operators(Proposition \ref{wave}),
there exists, for any $j$, a function $v_0^j$ in $H^1$ such that the corresponding solution $v^j$ of
\eqref{1.1} satisfies
$$\lim_{n\rightarrow\infty}\|e^{-it_{n}^{j}\Delta}\psi^{j}-v^j(t_n^j)\|_{H^1}=0.$$
Using the arguments in \cite{yuan}, we can show that for large $M$, the solution $u(x,t+\tau_n)$ of \eqref{1.1}
is close to the approximate solution $u_n\equiv\sum_{j=1}^Mv^j(x-x_n^j,t+t_n^j)$ for positive times.
More precisely, we obtain that $u_n$, which is the solution of  the approximate equation  $i\partial_tu_n+\Delta u_n+|u_n|^{p-1}u_n=e_n$ with
$e_n=|u_n|^{p-1}u_n-\sum_{j=1}^Mv^j(x-x_n^j,t+t_n^j)$, satisfies the following:\\
(1)~ For every $M>0,$ there exists $n_0= n_0(M)\in \mathbb N$
such that for all $n>n_0$, $\|u_n\|_{\dot H^{s_c}}\leq A$ with some large $A$ independent of $M$;\\
(2)~For every $M, \epsilon>0$, there exists $n_1=n_1(M,\epsilon)\in\mathbb N$
such that for all $n>n_1$,
$\|e_n\|_{\dot H^{s_c}}\leq \epsilon$;\\
(3)~ There exists $M_1=M_1(\epsilon)$ and $n_2=n_2(M_1)$ sufficiently large  such that for all
$n>n_2$,
$\|e^{it\Delta}(u(\tau_n)-u_n(0))\|_{\dot H^{s_c}}\leq\epsilon$.\\
Thus, by the perturbation theory (Proposition \ref{perturb}),
we obtain that $u(t+\tau_n)\approx u_n(t)$, which must scatter for positive time.
 This indeed yields a contradiction and so we obtain that there is only one nonzero profile.

Thus, now we have obtained that
\begin{align}\label{pd3}
u(x,\tau_n)=e^{-it_{n}^{1}\Delta}\psi^{1}(x-x_{n}^{1})+W_{n}^{1}(x),
\ \ \ \lim_{n\rightarrow+\infty}\|e^{it\Delta}W_{n}^{1}\|_{S(\dot{H}^{s_{c}})}=0.
\end{align}
We also claim that
\begin{align}\label{1}
\lim_{n\rightarrow}\|W^1_n\|_{H^1}=0.
\end{align}
Indeed, if not, we then obtain that for some $\epsilon>0$,
$M(e^{-it_{n}^{1}\Delta}\psi^j)^{\frac{1-s_c}{s_c}}E(e^{-it_{n}^{1}\Delta}\psi^j)
\leq M(Q)^{\frac{1-s_c}{s_c}}E(Q)-\epsilon,$
which, %%%%%%%%%%%%%%%%%%%%%%%%%%%%%%%%%%%%%%%%%%%%%%%%%%%%%%%%%%%%%%%%%%%%%%%%%%%%%%%%%%%%%%%%%%%%%%%%%%%%%%%%%%%%%%%%%%
 by similar arguments as above, implies that $u$ scatters, a contradiction.
% From \cite{yuan}, we know that a solution $v$ with initial data $v_0$ such that
%$\|v_0\|^{\frac{1-s_c}{s_c}}_2\|\nabla v_0\|_2<\|Q\|^{\frac{1-s_c}{s_c}}_2\|\nabla Q\|_2$
%scatters as $t\rightarrow\pm\infty$.

Finally, we claim that $t_n^1$ is bounded and thus converges up to extracting a subsequence.
Indeed, if $t_n^1\rightarrow+\infty$, then
$\|e^{it\Delta}u(\tau_n)\|_{S((-\infty,0],\dot{H}^{s_{c}})}
=\|e^{i(t-t_n^1)\Delta}\psi^1\|_{S((-\infty,0],\dot{H}^{s_{c}})}+o_n(1)
=\|e^{it\Delta}\psi^1\|_{S((-\infty,-t_n^1],\dot{H}^{s_{c}})}+o_n(1)\rightarrow0$
as $n\rightarrow+\infty$. This implies that $u$ scatters for negative time and,
by Proposition \ref{sd},
satisfies $\|u\|_{S((-\infty,\tau_n],\dot{H}^{s_{c}})}\rightarrow0$ as $n\rightarrow+\infty$.
Since $\tau_n>0$, we must have $u=0$, contradicting the assumptions.
Now, if $t_n^1\rightarrow-\infty$,
$\|e^{it\Delta}u(\tau_n)\|_{S([0,+\infty),\dot{H}^{s_{c}})}
%=\|e^{i(t-t_n^1)\Delta}\psi^1\|_{S([0,+\infty),\dot{H}^{s_{c}})}+o_n(1)
=\|e^{it\Delta}\psi^1\|_{S([-t_n^1,+\infty),\dot{H}^{s_{c}})}+o_n(1)\rightarrow0$,
showing that $u$ scatters for positive time. We get a contradiction again.
Thus we have proved the claim.

Consequently, the boundedness of $t_n^1$ combined with \eqref{1} immediately implies
 the  compactness of $K$.

\end{proof}

Now, for the solution $u$ of \eqref{1.1} satisfying \eqref{6.1}, we have got
the translation parameter $x(t)$ by Lemma \ref{l62}. Let the parameters
$X(t), \theta(t)$ and $\alpha(t)$ be defined for $t\in D_{\delta_0}$ as in
Section 4. Then by \eqref{4.4} and Lemma \ref{l42}, there exists some constant $C_0>0$ such that
for any $t\in D_{\delta_0}$,
$$\int_{|x-X(t)|\leq1}|\nabla u|^2+|u|^2\geq\int_{|x|\leq1}|\nabla Q|^2+|Q|^2-C_0\delta(t).$$
Taking $\delta_0$ smaller if necessary, we can assume that
for any $t\in D_{\delta_0}$,
$$\int_{|x+x(t)-X(t)|\leq1}|\nabla u(x+x(t))|^2+|u(x+x(t))|^2\geq\epsilon_0>0.$$
By the compactness of $\overline{K}$, we know that $|x(t)-X(t)|$ is bounded on $D_{\delta_0}$
and so we can modify $x(t)$ such that
\begin{align}\label{6.6}
x(t)=X(t),\ \ \ \forall t\in D_{\delta_0}
\end{align} and that
 $K$ defined by \eqref{6.2} remains precompact in $H^1$.
  As was discussed in \cite{holmer3} and \cite{nonradial},
it is classical that we can choose the function $x(t)$ to be continuous.
As a consequence, we have shown:
\begin{corollary}\label{c63}
Let $u$ be a solution of \eqref{1.1} satisfying the assumptions of Proposition \ref{p61}.
Then with the continuous function  $x(t)=X(t)$ with $X(t)$ defined by \eqref{4.4},
 the set $K$ defined by \eqref{6.2}
is precompact in $H^1$.
\end{corollary}

\begin{lemma}\label{l64}
Let $u$ be a solution of \eqref{1.1} satisfying the assumptions of Proposition \ref{p61}
and $x(t)$ defined by Corollary \ref{c63}.
Then
\begin{align}\label{6.7}
P(u)=Im\int \bar u\nabla udx=0.
\end{align}
Furthermore,
\begin{align}\label{6.8}
\lim_{t\rightarrow+\infty}\frac{x(t)}t=0.
\end{align}
\end{lemma}

\begin{proof}
The proof of \eqref{6.7} is easy. Indeed, assume $P(u)\neq0$ and consider the Galilean transformation of $u$
i.e.,
$w(x,t)=e^{ix\cdot\xi_0}e^{-it|\xi_0|^2}u(x-2\xi_0t,t)$. As was discussed in Remark \ref{p},
if we take $\xi_0=-P(u)/M(u)$ to minimize $E(w)$, then $M(w)=M(u)=M(Q)$, $E(w)<E(u)=E(Q)$ and immediately,
$M(w)^{\frac{1-s_c}{s_c}}E(w)<M(Q)^{\frac{1-s_c}{s_c}}E(Q)$.
By the result obtained in \cite{yuan},  this implies that $u$ must scatter in $H^1$, which contradicts the assumptions of the lemma
 concluding   \eqref{6.7}.

For the proof of \eqref{6.8}, one can refer to \cite{yuan}, and there is also a similar result in \cite{G1}.
\end{proof}

\subsection{Convergence in mean}
%By a similar argument as in \cite{holmer3}, we can obtain the following lemma:
\begin{lemma}\label{l65}
Let $u$ be a solution of \eqref{1.1} satisfying the assumptions of Proposition \ref{p61}.
Then
\begin{align}\label{6.9}
\lim_{T\rightarrow+\infty}\frac1T\int_0^T\delta(t)=0,
\end{align}
where $\delta(t)$ is defined by \eqref{delta}.
\end{lemma}
Before proving this lemma, we obtain from it the following corollary.
\begin{corollary}\label{c66}
Under the assumptions of Proposition \ref{p61}, there exists a sequence $t_n$ with
 $t_n+1\leq t_n$ such that $$\lim_{n\rightarrow+\infty}\delta(t_n)=0$$
as $t_n\rightarrow+\infty.$
\end{corollary}

Proof of Lemma \ref{l65}:

Let $\varphi\in C^\infty$ be defined as that in subsection 5.2:
$0\leq\varphi(r), \varphi''(r)\leq2$ and that
$\varphi(r)=r^2$ for $0\leq r\leq1$ while $\varphi(r)\equiv0$ for $r\geq2$.
We consider the localized variance $y_R(t)=\int R^2\varphi(\frac xR)|u(x,t)|^2dx$ again
and recall from subsection 5.2:
\begin{align}\label{6.11}
y'_R(t)=2RIm\int\bar u\nabla\varphi (\frac xR)\cdot\nabla u,
\ \ \ |y'_R(t)|\leq CR,
\end{align}
and
\begin{align}\label{6.13}
y''_R=\left(2N(p-1)-8\right)\left(\|\nabla Q\|^2_2-\|\nabla u\|^2_2\right)+A_R(u)
=\left(2N(p-1)-8\right)\delta(t)+A_R(u),
\end{align}
where
\begin{align}\label{6.12}
A_R(u(t))
=&4\sum_{j\neq k}\int\partial_j\partial_k\varphi(\frac xR)\partial_ju\partial_k\bar u
+4\sum_j\int\left(\partial^2_{x_j^2}\varphi(\frac xR)-2
\right)|\partial_ju|^2\\ \nonumber &
-\frac 1{R^2}\int\Delta^2\varphi(\frac xR)|u|^2
-\frac{2(p-1)}{p+1}\int\left(\Delta\varphi(\frac xR)-2N
\right)|u|^{p+1}.
\end{align}
By the properties of $\varphi$, we can obtain the the estimate for $A_R(u(t))$:
\begin{align}\label{6.14}
|A_R(u(t))|\leq C\int_{|x|\geq R}|\nabla u|^2+\frac1{R^2}|u|^2+|u|^{p}.
\end{align}
Let $x(t)=X(t)$ be as in $K$ defined by Corollary \ref{c63}.
By compactness of $K$, there exists $R_0(\epsilon)>0$
such that
\begin{align}\label{6.15}
\int_{|x-x(t)|\geq R_0(\epsilon)}|\nabla u|^2+|u|^2+|u|^{p}\leq\epsilon,\ \ \forall t\geq0.
\end{align}
Furthermore, by \eqref{6.8}, there exists $t_0(\epsilon)\geq0$ such that
\begin{align}\label{6.16}
|x(t)|\leq\epsilon t,\ \ \forall t\geq t_0(\epsilon).
\end{align}
Let $T\geq t_0(\epsilon)$ and $R=\epsilon T+R_0(\epsilon)+1$ for $t\in[t_0(\epsilon),T].$
Since $|x(t)|\leq\epsilon T$ and $\epsilon T+R_0(\epsilon)\leq R$, we get that
\begin{align}\label{6.17}
|A_R(u(t))|\leq &C\int_{|x|\geq R}|\nabla u|^2+\frac1{R^2}|u|^2+|u|^{p}\\ \nonumber
\leq &C\int_{|x-x(t)|+|x(t)|\geq R}|\nabla u|^2+|u|^2+|u|^{p}
\leq C\int_{|x-x(t)|\geq R_0(\epsilon)}|\nabla u|^2+|u|^2+|u|^{p}\leq\epsilon.
\end{align}
By \eqref{6.11} and \eqref{6.13},
$$\int_{t_0(\epsilon)}^T[4\delta(t)+A_R(u(t))]dt=
\int_{t_0(\epsilon)}^Ty''_R(t)dt
\leq|y'_R(t)|+|y'_R(t_0(\epsilon))|\leq CR.$$
\eqref{6.14} combined with \eqref{6.17} gives then, for some $C>0$ independent of $T$ and $\epsilon$,
$$\int_{t_0(\epsilon)}^T\delta(t)dt\leq C(R+T\epsilon)\leq C(R_0(\epsilon)+1+T\epsilon).$$
Thus, we obtain
$$\frac1T\int_0^T\delta(t)dt\leq \frac1T\int_0^{t_0(\epsilon)}\delta(t)dt+\frac CT(R_0(\epsilon)+1)+C\epsilon.$$
Passing to the limit first as $T\rightarrow+\infty$, then letting $\epsilon\rightarrow0$, we obtain \eqref{6.9}.

\ \ \ \ \ \ \ \ \ \ \ \ \ \ \ \ \ \ \ \ \ \ \ \ \ \ \ \ \ \ \ \ \ \ \ \ \ \ \ \ \ \ \ \ \ \ \ \ \ \ \ \ \ \ \ \ \ \ \ \ \ \ \ \ \ \ \
\ \ \ \ \ \ \ \ \ \ \ \ \ \ \ \ \ \ \ \ \ \ \ \ \ \ \ \ \ \ \ \ \ \ \ \ \ \ $\Box$

\subsection{Exponential convergence}
The aim of
 this subsection is to prove Proposition \ref{p61} by using the following
 Lemma \ref{l67} which  is a localized virial argument, and Lemma \ref{l68}, a precise
 control of the variations of the parameter $x(t)$.
%Both of  the two lemmas can be shown by similar arguments as in \cite{holmer3}, and so we omit the details.

\begin{lemma}\label{l67}
Let $u$ be a solution of \eqref{1.1} satisfying the assumptions of Proposition \ref{p61}
and $x(t)$ defined by Corollary \ref{c63}.
Then there exists a constant $C$ such that if $0\leq\sigma<\tau$
\begin{align}\label{6.18}
\int_\sigma^\tau\delta(t)dt\leq
C\Big(1+\sup_{\sigma\leq t\leq\tau}|x(t)|\Big)(\delta(\sigma)+\delta(\tau)).
\end{align}
\end{lemma}

\begin{proof}
For $R>0$ we consider the localized variance $y_R(t)=\int R^2\varphi(\frac xR)|u(x,t)|^2dx$.
Recall that
\begin{align}\label{6.19}
y'_R(t)=2RIm\int\bar u\nabla\varphi (\frac xR)\cdot\nabla u,
\ \ \
y''_R
=\left(2N(p-1)-8\right)\delta(t)+A_R(u),
\end{align}
where
$
A_R(u(t))$ is defined by \eqref{6.12}.

Now we show that if $\epsilon>0$,  there exists $R_\epsilon>0$ such that
\begin{align}\label{6.20}
\forall t\geq0,\ \ R\geq R_\epsilon(|x(t)|+1)\ \ \Rightarrow\ \  |A_R(u(t))|\leq \epsilon\delta(t).
\end{align}
The proof of the claim is divided in two cases.
When $\delta(t)$ is small, we consider $\delta_0$ as in section 4 and choose $0<\delta_1<\delta_0$
to be determined. For $t\in D_{\delta_1|}$, let $v=h+\alpha Q$ and then from \eqref{4.4}
and Lemma \ref{l42} we get that
\begin{align}\label{6.21}
u(x,t)=e^{i(t+\theta(t))}(Q(x-X(t))+v(x-X(t),t)),\ \ \ \ \|v\|_{H^1}\leq C\delta(t).
\end{align}
Note that fix $\theta_0$ and $X_0$, then  $A_R(e^{i\theta_0}e^{it}Q(\cdot+X_0))=0$ for any $R$ and $t$.
We obtain from the definition of $A_R$ that
\begin{align*}
&|A_R(u)|=|A_R(u)-A_R(e^{i\theta_0}e^{it}Q(\cdot+X_0))|\\
\leq &C\int_{|y+X(t)|\geq R}(|\nabla Q(y)||\nabla v(y)|+|\nabla v(y)|^2+Q(y)|v(y)|+|v(y)|^2+|v(y)|^{p+1}
)dy\\
 \leq &C\int_{|y+X(t)|\geq R}e^{-|y|}(|\nabla v(y)|+|v(y)|+|v(y)|^{p}
)dy+\int_{|y+X(t)|\geq R}(|\nabla v(y)|^2+|v(y)|^2+|v(y)|^{p+1}
)dy,
\end{align*}
Since $\|v\|_{H^1}\leq C\delta(t)$ by Lemma \ref{l42}, then
choosing $R_0$ sufficiently large and $\delta_1$ small enough, we obtain
\begin{align}\label{6.23}
R\geq|X(t)|+ R_0,\ \ \delta(t)\leq\delta_1, \ \ \Rightarrow\ \  |A_R(u(t))|
\geq \epsilon\delta(t).
\end{align}
Recall that by \eqref{6.6}, $x(t)=X(t)$ on $D_{\delta_0}$ and \eqref{6.23} implies \eqref{6.20} for $\delta(t)<\delta_1$.

In the case $\delta(t)\geq\delta_1$, there exists some $C>0$ such that for any $t\geq0$,
\begin{align*}
|A_R(u)|
\leq &C\int_{|x|\geq R}(|\nabla u(y)|^2+|u(y)|^2+|u(y)|^{p+1}
)dx\\
 \leq &C\int_{|x-x(t)|\geq R-|x(t)|}(|\nabla u(y)|^2+|u(y)|^2+|u(y)|^{p+1}
)dx.
\end{align*}
By the compactness of $K$, there exists $R_1>0$ such that
\begin{align}\label{6.24}
R\geq |x(t)|+R_1,\ \ \delta(t)\geq\delta_1, \ \ \Rightarrow\ \  |A_R(u(t))|
\geq \epsilon\delta_1\leq\epsilon\delta(t).
\end{align}
Finally, we have proved \eqref{6.20}.

By \eqref{6.19} and  \eqref{6.20}, we obtain that there exists $R^*>0$ such that
$$R\geq R^*(|x(t)|+1)\ \ \Rightarrow\ \  y''_R(t)\geq (N(p-1)-4)\delta(t).$$
Let $R=R^*(\sup_{\sigma\leq t\leq\tau}|x(t)|+1)$,
 we obtain
\begin{align}\label{6.25}
(N(p-1)-4)\int_\sigma^\tau\delta(t)dt\leq\int_\sigma^\tau y''_R(t)dt=y'_R(\tau)-y'_R(\sigma).
\end{align}
If $\delta(t)<\delta_0$, by \eqref{6.19} and \eqref{6.21}, then
\begin{align*}
y'_R(t)=&2RIm\int\bar v(z)\nabla\varphi (\frac {z+X(t)}R)\cdot\nabla Q(z)\\
&+2RIm\int Q(z)\nabla\varphi (\frac {z+X(t)}R)\cdot\nabla v(z)
+2RIm\int\bar v(z)\nabla\varphi (\frac {z+X(t)}R)\cdot\nabla  v(z),
\end{align*}
which implies by Lemma \ref{l42} that $|y'_R(t)|\leq CR(\delta(t)+\delta^2(t))\leq R\delta(t)$.
On the other hand, when $\delta(t)\geq\delta_0$, the above inequality  follows by straightforward estimate.
Hence by \eqref{6.25} and the choice $R=R^*(\sup_{\sigma\leq t\leq\tau}|x(t)|+1)$, we obtain \eqref{6.18} and complete our proof.

\end{proof}

The following lemma is to control of the variations of $x(t)$.
\begin{lemma}\label{l68}
There  exists a constant $C$ such that for any $\sigma,\tau>0$
with $\sigma+1\leq\tau$,
\begin{align}\label{6.26}
|x(\tau)-x(\sigma)|\leq C\int_\sigma^\tau\delta(t)dt.
\end{align}
\end{lemma}
The proof of the lemma  can be found  in \cite{holmer3} (Lemma 6.8 there).

Now, we are ready to show Proposition \ref{p61}. \\
Proof of Proposition \ref{p61}:
Consider the sequence $t_n$ given by Corollary \ref{c66} and so
$t_n\rightarrow+\infty$, $t_n+1\leq t_n$, and $\delta(t_n)\rightarrow0$.
By Lemma \ref{l67} and Lemma \ref{l68}, there exists some $C_0>0$ %and $N_0\in\mathbb N$
 such that
$$\forall n>N_0,\ \ 1+t_{N_0}\leq t_n\ \ \Rightarrow\ \
|x(t_{N_0})-x(t)|\leq C_0(1+\sup_{[t_{N_0},t_n]}|x(t)|)[\delta(t_{N_0})+\delta(t_n)].$$
We choose $t$ such that $|x(t)|=\sup_{[t_{N_0}+1,t_n]}|x(s)|$ and then
$$\sup_{[t_{N_0}+1,t_n]}|x(s)|\leq C(N_0)+C_0(1+\sup_{[t_{N_0}+1,t_n]}|x(s)|)[\delta(t_{N_0})+\delta(t_n)]$$
with $C(N_0)=|x(N_0)|+C_0\sup_{[t_{N_0},t_{N_0}+1]}|x(s)|$.
Fixing  $N_0$ large enough, we can assume $\delta(t_{N_0})+\delta(t_n)\leq1$
and $C_0\delta(t_{N_0})\leq\frac12$. Thus, for $t_n\geq t_{N_0}+1$,
$$\frac12\sup_{[t_{N_0}+1,t_n]}|x(s)|\leq C(N_0)+\frac12+C_0(1+\sup_{[t_{N_0}+1,t_n]}|x(s)|)\delta(t_n).$$
Letting $n\rightarrow+\infty$, since $\delta(t_n)\rightarrow0$, we obtain
that $|x(t)|$ is bounded on $[t_{N_0}+1,+\infty)$. By continuity, we finally obtain the boundedness
of $|x(t)|$ on $[0,+\infty)$.

Lemma \ref{l67} combined with the boundedness of $x(t)$ gives that
for any $\sigma,\tau>0$ and $0\leq\sigma\tau$, $\int_\sigma^\tau\delta(t)dt\leq C(\delta(\sigma)+\delta(\tau)).$
If we take $\tau=t_n$  and let $n\rightarrow+\infty$, we obtain that $\int_0^\infty\delta(t)dt<\infty$.
Thus, for any $\sigma>0$, $\int_\sigma^\infty\delta(t)dt\leq C\delta(\sigma),$
 By Gronwall's Lemma, we obtain that there exist $C,c>0$
$$\int_\sigma^\infty\delta(t)dt\leq Ce^{-c\sigma}.$$
Since $\sigma>0$ is arbitrary, we have concluded the proof of Proposition \ref{p61}
again using Lemma \ref{l44}. \ \ \ \ \ \ \ \ \ \ \ \ \ \ \ \ \ \ \ \ \ \ \ \ \ \ \ \ \ \ \ \ \ \ \ \ \ 
 \ \ \ \ \ \ \ \ \ \ \ \  \ \ \ \ \ \ \ \ \ \ \ \  \ \ \ \ \ \ \ \ \ \ \ \  \ \ \ \ \ \ \ \ \ \ \ \  \ \ \ \ \ \ \ \ \ \ \ \ 
 \ \ \ \ \ \ \ \ \ \ \ \  \ \ \ \ \ \ \ \ \ \ \ \  \ \ \ \ \ \ \ \ \ \ \ \  \ \ \ \ \ \ \ \ \ \ \ \  \ \ \ \ \ \ \ \ \ \ \ \  $\Box$

\subsection{Scattering of $Q^-$ for negative times}
In the final subsection,
we conclude the proof of Theorem \ref{th2} by showing that
 the special solution $Q^-$ scatters as $t\rightarrow-\infty.$
If not, we apply the argument of above subsections to the solution
$Q^-$ and $\overline{Q^-}(x,-t)$ of \eqref{1.1} and obtain a parameter $x(t)$
defined for $t\in \mathbb R$ such that
$\tilde K=\{Q^-(\cdot+x(t),t),t\in\mathbb R\}$ has a compact closure in $H^1$.
By the argument at the end of Subsection 6.3, $x(t)$ is bounded and $\delta(t)$
tends to $0$ as $t\rightarrow\pm\infty$.
A simple adjustment of Lemma \ref{l67} implies that if $-\infty<\sigma\leq\tau<
+\infty$ then
$$\int_\sigma^\tau\delta(t)dt\leq
C\Big(1+\sup_{\sigma\leq t\leq\tau}|x(t)|\Big)(\delta(\sigma)+\delta(\tau))
\leq C(\delta(\sigma)+\delta(\tau)).$$
Letting $\sigma\rightarrow-\infty$ and $\tau\rightarrow+\infty$, we obtain then
$\int_\mathbb R\delta(t)dt=0.$ Thus $\delta(t)=0$ for all $t$
which contradicts the assumption $\|\nabla u_0\|_2<\|\nabla Q\|_2$.

\section{uniqueness}%%%%%%%%%%%%%%%%%%%%%%%%%%%%%%%%%%%%%%%%%%%%%%%%%%%%%%%%%%%%%%%%%%%%%%%%%%%%%%%%%%%%%%%%%%%%%%%%%%%%%%%%%%%重新交代证明的不同和p带来的难点
We will finally conclude the proof of Theorem \ref{th3} in this section. The main point is to show the following
 uniqueness result. We want to point out that our arguments in this section are different from that in \cite{holmer3},
which are indeed  invalid for our general $L^2$-supercritical case.
\begin{proposition}\label{p71}
Let $u$ be a solution of \eqref{1.1} defined on $[t_0,+\infty)$
such that $E(u)=E(Q)$, $M(u)=M(Q)$. Assume that there exist $c,C>0$ such that
for any $t\geq t_0$,
\begin{align}\label{7.1}
\|u-e^{i(1-s_c)t}Q\|_{H^1}\leq Ce^{-ct}.
\end{align}
Then there exists $A\in\mathbb R$ such that $u=U^A$, where
$U^A$ is defined by Proposition \ref{UA}.
\end{proposition}

The proof of Theorem \ref{th3} is divided into three parts.
In subsection 7.1, we analyze
 the linearized equation
and the spectral
properties of $\mathcal L$ defined by \eqref{linea},
using which we conclude the proof of Proposition \ref{p71} in
subsection 7.2. Finally, in subsection 7.3, we finish the proof of Theorem \ref{th3}.

Throughout this section, we  often use the following integral summation argument
introduced in  \cite{M2} (Claim 5.8 there):
\begin{lemma}\label{summation}
Let $t_0>0$, $p\geq1$, $a_0\neq0$ and $E$ is a normed vector space. If
 $f\in L_{loc}^p([t_0,\infty);E)$
satisfies  that
$$\exists\tau_0>0, C_0>0,\ \ \forall t\geq t_0,\ \ \ \|f\|_{L^p([t,t+\tau_0);E)}\leq C_0e^{a_0t},$$
 then, for $t\geq t_0$, we have
$$\|f\|_{L^p([t,\infty);E)}\leq \frac{C_0e^{a_0t}}{1-e^{a_0\tau_0}},
\ \ \ if\ \ a_0<0;\ \ \
\|f\|_{L^p([t_0,t);E)}\leq \frac{C_0e^{a_0t}}{1-e^{-a_0\tau_0}},
\ \ \ if\ \ a_0>0.$$
\end{lemma}

\subsection{Exponentially small solutions of the linearized equation}
Set $\tilde r=p+1$ and $\frac 2{\tilde q}=N(\frac12-\frac1{\tilde r})$.
We consider $$v\in C^0([t_0,+\infty),H^1),\ \ g\in L^{\tilde q}([t_0,+\infty),W^{1,\tilde r})$$
such that
\begin{align}\label{7.2}
\partial_tv+\mathcal Lv=g,\ \ (x,t)\in\mathbb R^N\times(t_0,+\infty),
\end{align}
\begin{align}\label{7.3}
\|v(t)\|_{H^1}\leq Ce^{-\gamma_1t},\ \ \|g(t)\|_{L^{\tilde q'}([t,+\infty),W^{1,\tilde r'})}\leq Ce^{-\gamma_2t},%%%%%%%%%%%%%%%%%%%%%%%%%%%检查
\end{align}
where $0<\gamma_1<\gamma_2$.

The following self-improving estimate is important for our analysis.
%We prove it with a different method from that in \cite{holmer3}.
\begin{lemma}\label{l72}
Under the above assumptions,\\
(a) if $\gamma_2\leq e_0$, then $\|v(t)\|_{H^1}\leq Ce^{-\gamma_2^-t}$,\\
(b) if $\gamma_2> e_0$, then there exists $A\in\mathbb R$ such that
$v(t)=Ae^{-e_0t}\mathcal Y_++w(t)$
with $\|w(t)\|_{H^1}\leq Ce^{-\gamma_2^-t}$.
\end{lemma}

\begin{proof}
We first recall the quadratic form $\Phi$ defined by \eqref{Phi} and the associated bilinear form $B$ by \eqref{B}.
 We have known that $B(Q_j,h)=0$ and $\|Q_j\|_2=1$ for
any $h\in H^1$ and $j=0,\cdots, N$,
where we denote $$Q_0\equiv \frac{iQ}{\|Q\|_2},\ \ \ Q_j\equiv\frac{\partial_jQ}{\|\partial_jQ\|_2}.$$
By definition, we can obtain $\Phi(\mathcal Y_+)=\Phi(\mathcal Y_-)=0$. Furthermore, we assert that
$B(\mathcal Y_+,\mathcal Y_-)\neq0$. In fact,
if $B(\mathcal Y_+,\mathcal Y_-)=0$, then $B$ and $\Phi$ would be identically $0$
on $span\{\partial_jQ, iQ, \mathcal Y_+, \mathcal Y_-, j=1,\cdots,N\}$ which is of dimension $N+3$.
But $\Phi$ is, by Proposition \ref{coercivity}, positive on $G_\perp$ which is of codimension $N+2$,
 yielding a contradiction by Courant's min-max principle.
Thus,  we can normalize the eigenfunctions $\mathcal Y_+,\mathcal Y_-$
such that $B(\mathcal Y_+,\mathcal Y_-)=1$. Then  $h\in G'_\perp$
is equivalent to
 $$(Q_j,h)=0,\ \ \ B(\mathcal Y_+,h)=B(\mathcal Y_-,h)=0\ \ \ \forall j=0,\cdots,N.$$

Now we decompose $v(t)$ as
\begin{align}\label{v}
v(t)=\alpha_+(t)\mathcal Y_++\alpha_-(t)\mathcal Y_-+\sum_{j=0}^N\beta_j(t)Q_j+v_\perp(t),
\ \ \ \ v_\perp\in G'_\perp,
\end{align}
where
\begin{align}\label{7.8}
&\beta_j(t)=(v(t),Q_j)-\alpha_+(t)(\mathcal Y_+,Q_j)-\alpha_-(t)(\mathcal Y_-,Q_j),\\ \nonumber
&\alpha_+(t)=B(v(t),\mathcal Y_-),\ \ \ \alpha_-(t)=B(v(t),\mathcal Y_+).
\end{align}

 Step 1.  By differentiating the equation 
 on the coefficients \eqref{7.8} and note that $B(\mathcal Lv,v)=0$, we  obtain  that
  \begin{align}\label{7.7'}
\frac d{dt}(e^{-e_0t}\alpha_-(t))=e^{-e_0t}B(g,\mathcal Y_+),\ \ \
\frac d{dt}(e^{e_0t}\alpha_+(t))=e^{e_0t}B(g,\mathcal Y_-),
\end{align}
\begin{align}\label{7.8'}
\beta'_j(t)=(v_t-\alpha'_+\mathcal Y_+-\alpha'_-\mathcal Y_-,Q_j)=&\Big(
g-B(g,\mathcal Y_-)\mathcal Y_+-B(g,\mathcal Y_+)\mathcal Y_--\mathcal Lv_\perp,Q_j\Big)\\ \nonumber
\equiv &(\tilde v,Q_j),
\end{align}
and
\begin{align}\label{7.9'}
\frac d{dt}\Phi(v(t))=2B(g,v).
\end{align}

Step 2.
We now show the following estimates :
 \begin{align}\label{7.7}
|\alpha_-(t)|\leq Ce^{-\gamma_2t},
\end{align}
 \begin{align}\label{7.10}
|\alpha_+(t)|\leq Ce^{-\gamma^-_2t},\ \ \ if\ \ \
\gamma_2\leq e_0\ \ or\ \ e_0\leq\gamma_1
\end{align}
and there exits $A\in\mathbb R$ such that
 \begin{align}\label{7.11}
|\alpha_+(t)-Ae^{-e_0t}|\leq Ce^{-\gamma_2t},\ \ \ if\ \ \
\gamma_2> e_0.
\end{align}
By definition \eqref{B},
\begin{align}\label{B1}
&2B(g,\mathcal Y_+)
 =\int(L_+g_1)\mathcal Y_1+\int(L_-g_2)\mathcal Y_2\\ \nonumber
 =&-\int g_1\Delta\mathcal Y_1+\int(1-s_c) g_1\mathcal Y_1-\int pQ^{p-1} g_1\mathcal Y_1
 -\int g_2\Delta\mathcal Y_2+\int(1-s_c) g_2\mathcal Y_2-\int Q^{p-1} g_2\mathcal Y_2
\end{align}
Hence, for any time interval $I$ with $|I|<\infty$, we have
\begin{align*}
\int_I|B(g,\mathcal Y_\pm)|dt
\leq C |I|^{\frac1{\tilde q'}}\|g\|_{L^{\tilde q}(I,L^{\tilde r})}\|\mathcal Y_\pm\|_{W^{2,\tilde r}},
\end{align*}
which, together with \eqref{7.3}, implies that
\begin{align*}
\int_t^{t+1}|e^{-e_0s}B(g(s),\mathcal Y_+)|ds
\leq Ce^{-e_0t}e^{-\gamma_2t}.
\end{align*}
By Lemma \ref{summation}, we have then
\begin{align}\label{B2}
\int_t^{\infty}|e^{-e_0s}B(g(s),\mathcal Y_+)|ds
\leq Ce^{-e_0t}e^{-\gamma_2t}.
\end{align}
From \eqref{7.3} we know that $e^{-e_0t}\alpha_-(t)$ tends to 0 as $t$ goes to infinity.
Integrating the equation on $\alpha_-$ in  \eqref{7.7'} on $[t,+\infty)$, we obtain that
$|\alpha_-(t)|\leq Ce^{-\gamma_2t}$ showing \eqref{7.7}.

Now, we prove \eqref{7.10}.
In the case $e_0<\gamma_1$, by \eqref{7.3}, we have that
$e^{e_0t}\alpha_+(t)$ tends to 0 as $t$ goes to infinity.
By similar estimates as \eqref{B2}, we also have that
\begin{align*}
\int_t^{\infty}|e^{e_0s}B(g(s),\mathcal Y_-)|ds
\leq Ce^{e_0t}e^{-\gamma_2t}.
\end{align*}
Integrating the equation on $\alpha_+$ in  \eqref{7.7'} on $[t,+\infty)$, we obtain that
$|\alpha_+(t)|\leq Ce^{-\gamma_2t}$.
In the case $\gamma_1\leq e_0<\gamma_2$, also by \eqref{7.3},
\begin{align*}
\int_t^{t+1}|e^{e_0s}B(g(s),\mathcal Y_-)|ds
\leq Ce^{e_0t}e^{-\gamma_2t},
\end{align*}
which together with Lemma \ref{summation} gives that
\begin{align*}
\int_{t_0}^{\infty}|e^{e_0s}B(g(s),\mathcal Y_-)|ds
\leq Ce^{e_0{t_0}}e^{-\gamma_2{t_0}}<\infty.
\end{align*}
By \eqref{7.7'}, $e^{e_0t}\alpha_+(t)$ satisfies the Cauchy criterion as $t\rightarrow+\infty$.%%%%%%%%%%%%%%%%%%%%%%%%%%%%%%%%%%%%%%%%%%%%%思考
Then, there exists $A$ such that
$\lim_{t\rightarrow+\infty}e^{e_0t}\alpha_+(t)=A$
and $$|\alpha_+(t)-A|\leq Ce^{e_0t}e^{-\gamma_2t},$$
showing \eqref{7.11}.

In the case $\gamma_1<\gamma_2\leq e_0$, integrating the equation on $\alpha_+$ in \eqref{7.7'} on $[0,t]$,
we obtain that
$$\alpha_+(t)=e^{-e_0t}\alpha_+(0)+e^{-e_0t}\int_0^te^{e_0s}B(g,\mathcal Y_-)ds,$$
which, by \eqref{7.3}, yields that
$$\left|\int_0^te^{e_0s}B(g,\mathcal Y_-)ds\right|\leq
\begin{cases}
Ce^{e_0t}e^{-\gamma_2t},\ \ &\gamma_2<e_0,\\
Ct,\ \ &\gamma_2=e_0.
\end{cases}$$
This shows \eqref{7.10} in this case.

In the following steps, we prove
 Lemma \ref{l72}  under the conditions \eqref{7.7}, \eqref{7.10} and \eqref{7.11}.

 Step 3.
 We first do with the case $\gamma_2\leq e_0$ or $\gamma_2>e_0$ and $A=0$.
 By step 2, we have got in this case that
 \begin{align}\label{7.13}
|\alpha_+(t)|+|\alpha_-(t)|\leq Ce^{-\gamma^-_2t},\ \ \ \forall t\geq t_0.
\end{align}
Since $$\int_t^{t+1}B(g,v)ds\leq Ce^{-(\gamma_1+\gamma_2)t},$$
we have, by Lemma \ref{summation}, that
$$\int_t^{\infty}B(g,v)ds\leq Ce^{-(\gamma_1+\gamma_2)t}.$$
By \eqref{7.9'} and $|\Phi(v(t))|\leq C\|v(t)\|^2_{H^1}\rightarrow0$ as $t\rightarrow+\infty$,
we have that  $|\Phi(v(t))|\leq Ce^{-(\gamma_1+\gamma_2)t}$.
Note that $\Phi(v)=B(v,v)=B(v_\perp,v_\perp)+2\alpha_+\alpha_-$, so we obtain from
 Proposition \ref{coercivity} and \eqref{7.13} that
  \begin{align}\label{7.19'}
\|v_\perp\|_{H^1}^2\leq C|B(v_\perp,v_\perp)|\leq Ce^{-(\gamma_1+\gamma_2)t}.
\end{align}

Now we turn to estimate the decay of $\beta_j$.
By \eqref{7.8} and the above step, we know that  $|\beta_j(t)|\rightarrow0$
as $t\rightarrow+\infty$.
Moreover, by the notation of $\tilde v$,
 \begin{align*}
 \int_t^{t+1}|(\tilde v,Q_j)|ds\leq C\Big(e^{-\gamma_2}+\int_t^{t+1}|(\mathcal Lv_\perp,Q_j)|ds \Big)
 \leq C\Big(e^{-\gamma_2}+\|v_\perp\|_{L^\infty H^1} \Big)
 \leq Ce^{-(\frac{\gamma_1+\gamma_2}{2})t}.
\end{align*}
Thus by \eqref{7.8'} and Lemma \ref{summation}, we obtain that
  \begin{align}\label{beta}
|\beta_j(t)|\leq Ce^{-(\frac{\gamma_1+\gamma_2}{2})t}.
\end{align}
Thus,  we have got that $v$ and $g$ satisfy the assumption \eqref{7.3}
with $\gamma_1$ replaced by $\gamma'_1=\frac{\gamma_1+\gamma_2}{2}$.
Finally, by an iteration argument,
we can  obtain  that
\begin{align}\label{7.17}
\|v\|_{H^1}\leq Ce^{-\gamma^-_2t}
\end{align}
in the case $\gamma_2\leq e_0$ or $\gamma_2>e_0$ and $A=0$.

 Step 4. We finish the proof of Lemma \ref{l72} by dealing with the case $\gamma_2>e_0$ and
 $A\neq0$.
In this case, it suffices to assume $\gamma_1\leq e_0$ since, otherwise, we can take $A=0$ by Step 2.
Let $\tilde v(t)\equiv v(t)-Ae^{-e_0t}\mathcal Y_+$, it holds that
$$\partial_t\tilde v(t)+\mathcal L\tilde v(t)=g(t),\ \ \ \|\tilde v\|_{H^1}\leq Ce^{-\gamma_1t}.$$
We consider $\tilde\alpha_+(t)=B(\tilde v(t),\mathcal Y_-)$, which is the corresponding coefficient of $\mathcal Y_+$
in the decomposition of $\tilde v$.
By the decomposition of $ v$, we get that $\tilde\alpha_+(t)=B(v(t)-Ae^{-e_0t}\mathcal Y_+,\mathcal Y_-)=\alpha_+(t)-Ae^{-e_0t}$.
Thus by
 \eqref{7.11}, we have that
$|\tilde\alpha_+(t)|\leq Ce^{-\gamma_2^-t}$, turning back to the case discussed in Step 3.
As a consequence,
$$\|v(t)-Ae^{-e_0t}\mathcal Y_+\|_{H^1}=\|\tilde v(t)\|_{H^1}\leq Ce^{-\gamma^-_2t},$$
which conclude the proof of Lemma \ref{l72}.

\end{proof}

\subsection{Uniqueness}

We prove Proposition \ref{p71}. For $u$ satisfies the hypothesis, we write $u=e^{i(1-s_c)t}(Q+h)$.

Step 1. We show that if $e_0^-$ is any positive number such that $e_0^-<e_0$,
then for any $t\geq t_0$,
\begin{align}\label{7.18}
\|h(t)\|_{H^1}\leq Ce^{-e_0^-t}.
\end{align}
Indeed, from the equation \eqref{linea2},  by Strichartz's estimate and \eqref{eS},
we know from the local existence theory that, for any $(q,r)\in\Lambda_0$,
$$\|h\|_{L^q([t_0,\infty);W^{1,r})}\leq C\|h(t_0)\|_{H^1}
\leq Ce^{-ct}.$$
This, in turn, implies that  $\|R(h)\|_{L^{\tilde q'}([t_0,\infty);W^{1,\tilde r'})}\leq Ce^{-2ct}$
satisfying
the assumptions of Lemma \ref{l72} with
$\gamma_1=c, \gamma_2=2c.$
If $2c>e_0$, the proof is complete; otherwise, we get by Lemma \ref{l72}  that $\|h(t)\|_{H^1}\leq Ce^{-2c^-t}$
and then \eqref{7.18} follows by iteration arguments.

Step 2. Consider the solution $U^A$ constructed in Proposition \ref{UA}
and write $U^A=e^{i(1-s_c)t}(Q+h^A)$.
We show that there exists $A\in \mathbb R$ such that for all $\gamma>0$,
there exists $C>0$ such that for any $t\geq t_0$,
\begin{align}\label{7.19}
\|h(t)-h^A(t)\|_{H^1}\leq Ce^{-\gamma t}.
\end{align}
According to Step 1, $h$ fulfills the assumptions of Lemma \ref{l72} with
$\gamma_1=e_0^-, \gamma_2=2e_0^-$. Thus, there exists $A\in\mathbb R$ such that
\begin{align}\label{7.20}
\|h(t)-Ae^{-e_0t}\mathcal Y\|_{H^1}\leq Ce^{-2e_0^- t}.
\end{align}
By the asymptotic development of $h^A$ obtained in Section 3,
$$\|h^A(t)-Ae^{-e_0t}\mathcal Y\|_{H^1}\leq Ce^{-2e_0^- t}.$$
Thus, \eqref{7.20} implies \eqref{7.19} for any $\gamma<2e_0.$
We next show that if \eqref{7.19} holds for some $\gamma>e_0$, then
it holds for $\gamma'=\gamma+\frac12 e_0$.
In fact, since $h-h^A$ solves the equation
$$\partial_t(h-h^A)+\mathcal L(h-h^A)=R(h)-R(h^A).$$
Again from the local well-posedness theory, for any admissible pair $(q,r)\in\Lambda_0$,
$$\|h-h^A\|_{L^q([t_0,\infty);W^{1,r})}\leq C\|h(t_0)-h^A(t_0)\|_{H^1}
\leq Ce^{-\gamma t},$$
which in turn gives by \eqref{eR} that
$\|R(h)-R(h^A)\|_{L^{\tilde q}([t_0,\infty);W^{1,\tilde r})}\leq Ce^{-(e_0+\gamma)t}$.%%%%%%%%%%%%%%%%%%%%%%%%%%%%%%%%%%%%%%%%%%%%%%%%%%%%%%%%检查
Thus,  $h-h^A$ fulfills the conditions of Lemma \ref{l72} with $\gamma_1=\gamma, \gamma_2=\gamma+e_0$.
Then we get \eqref{7.19} with $\gamma$ replaced by $\gamma+\frac12 e_0$.
By iteration, \eqref{7.19} holds for any $\gamma>0$. %Using this with $\gamma=(k_0+1)E_0$ where $k_0$ is given in the proof of
Thus, we have obtained that $\|u-U^A\|_{H^1}\leq Ce^{-\gamma t}$ for any $\gamma>0$ and any $t\geq t_0$.
By the definition of $U^A$,  we obtain
\begin{align*}
\|u-e^{i(1-s_c)t}(Q+\mathcal V_{k_0}^A(t))\|_{H^b}\leq C e^{-(k_0+\frac12)e_0t}
\end{align*}
with $\mathcal  V_{k_0}^A$ and $ k_0$  constructed in Proposition \ref{ZA}. Then,
by the uniqueness argument in the proof of Proposition \ref{UA}, we get then $u=U^A$,
concluding  Proposition \ref{p71}.
\ \ \ \ \ \ \  \ \ \ \ \ \ \ \ \ \ \ \ \ \ \ \ \ \ \ \ \ \ \ \ \ \ \ \ \ \ \ \ \ \ \ \ \ \ \ \ \ \ \ \ \ \ \ \ \ \ \ \ \ \ \ \ \ \
 \ \ \ \ \ \ \ \ \ \ \ \ \ \ \ \ \ \ \ \ \ \ \ \ \ \ \ \ $\Box$
\subsection{Proof of the classification result}

We finish the proof of Theorem \ref{th3} in this subsection.
We first claim that if $A\neq0$, $U^A=Q^+$ for $A>0$ or $U^A=Q^-$ for $A<0$
up to a translation in time and a multiplication by a complex number of modulus 1.
Indeed, by \eqref{3.1},
\begin{align}\label{7.23}
Q^\pm(t)=e^{i(1-s_c)t}Q\pm e^{-e_0t_0}e^{(i-e_0)t}\mathcal Y_++O(e^{-2e_0t})\ \ in \ \ H^1.
\end{align}
Fix $A>0$. Let $t_1=-t_0-\frac 1{e_0}\log A$ such that $e^{-e_0(t_0+t_1)}=A$.
By \eqref{3.1} and \eqref{7.23}, we obtain that
\begin{align}\label{7.24}
e^{-it_1}Q^+(t+t_1)=e^{i(1-s_c)t}Q+ e^{-e_0(t_0+t_1)}e^{(i-e_0)t}\mathcal Y_++O(e^{-2e_0t})
=U^A+O(e^{-2e_0t})\ \ \ in \ \ H^1.
\end{align}
On the other hand, $e^{-it_1}Q^+(t+t_1)-e^{i(1-s_c)t}Q\rightarrow0$ exponentially in $H^1$ as $t\rightarrow+\infty$.
By Proposition
\ref{p71}, there exists $\tilde A$ such that $e^{-it_1}Q^+(t+t_1)=U^{\tilde A}.$
By \eqref{7.24}, we have $\tilde A=A$ and thus $U^A=e^{-it_1}Q^+(t+t_1)$. The case $A<0$ can be shown similarly.

Let $u$ satisfy the hypothesis of Theorem \ref{th3}. We rescale $u$ such that
$E(u)=E(Q),\ \ M(u)=M(Q).$

If $\|\nabla u_0\|_2=\|\nabla Q\|_2$, by the variational characterization of $Q$,
 $u=e^{i(1-s_c)t}Q$ up to the symmetries of the equation which yields case (b).

If $\|\nabla u_0\|_2<\|\nabla Q\|_2$ and assume that $u$ does not scatter for both positive and
negative times. Replacing $u(x,t)$ by $\bar u(x,-t)$ if necessary, we may assume $u$
does not scatter for positive times. By Proposition \ref{p61}, there exist
$\theta_0\in\mathbb R, x_0\in\mathbb R^N$ and $c,C>0$ such that
$\| u(t)-e^{i(1-s_c)t+i\theta_0}Q(\cdot-x_0)\|_{H^1}\leq Ce^{-ct}$ for $t>0$.
Hence, $v(x,t)=e^{-i\theta_0}u(x+x_0,t)$ satisfies the assumptions of Proposition \ref{p71}, which shows that
$v=U^A$ for some $A$. Since  $\|\nabla u_0\|_2<\|\nabla Q\|_2$, by Remark \ref{label},
the parameter $A$
should be negative. Thus, from the arguments in the first paragraph 
of this subsection, we get that $v=Q^-$ up to the symmetries of the equation, yielding case (a).

We can show case (c)  similar to  case (a) in view of Proposition \ref{p51}
and Proposition \ref{p71} and conclude  the proof of Theorem \ref{th3}.
\ \ \ \ \ \ \  \ \ \ \ \ \ \ \ \ \ \ \ \ \ \ \ \ \ \ \ \ \ \ \ \ \ \ \ \ \ \ \ \ \ \ \ \ \ \ \ \ \ \ \ \ \ \ \ \ \ \ \ \ \ \ \ \ \
$\Box$

%\appendix
%\section{coercivity properties of the quadratic form}

%\section{}

%{\bf Acknowledgements:} D.Cao is grateful to L.Caffarelli forbringing the problem studied in this paper to his attention. Both
%D.Cao and S.Peng were supported by the Key Project ofNSFC(No.10631030). D.Cao was also supported partially by Science
%Fund for Creative Research Groups of Natural Science Foundation ofChina (No.10721101). S.Peng was also supported by the Program for
%New Century Excellent Talents in University (No.07-0350). S.Yan was partiallysupported by ARC in Australia.

\end{document}